\documentclass[12pt]{amsart}
\usepackage{amscd,verbatim}
\usepackage{amssymb}
\usepackage[all]{xy}
\usepackage{color}

\newcommand{\A}{\mathbf{A}}

\newcommand{\G}{\mathbf{G}}
\renewcommand{\L}{\mathbb{L}}
\renewcommand{\P}{\mathbf{P}}

\newcommand{\Z}{\mathbf{Z}}
\newcommand{\sA}{\mathcal{A}}
\newcommand{\sB}{\mathcal{B}}
\newcommand{\sF}{\mathcal{F}}
\newcommand{\sG}{\mathcal{G}}
\newcommand{\sH}{\mathcal{H}}
\newcommand{\sO}{\mathcal{O}}

\newcommand{\bH}{\mathbb{H}}
\newcommand{\fp}{\mathfrak{p}}
\newcommand{\Mod}{\hbox{--}\operatorname{Mod}}
\newcommand{\Span}{\operatorname{\mathbf{Span}}}
\newcommand{\Cor}{\operatorname{\mathbf{Cor}}}
\newcommand{\Mack}{\operatorname{\mathbf{Mack}}}
\newcommand{\HI}{\operatorname{\mathbf{HI}}}
\newcommand{\PST}{{\operatorname{\mathbf{PST}}}}
\newcommand{\NST}{\operatorname{\mathbf{NST}}}
\newcommand{\DM}{\operatorname{\mathbf{DM}}}
\newcommand{\Alb}{\operatorname{Alb}}
\newcommand{\Hom}{\operatorname{Hom}}
\newcommand{\uHom}{\operatorname{\underline{Hom}}}
\newcommand{\Ext}{\operatorname{Ext}}
\newcommand{\Ker}{\operatorname{Ker}}
\newcommand{\IM}{\operatorname{Im}}
\newcommand{\Coker}{\operatorname{Coker}}
\newcommand{\uCH}{\underline{CH}}
\newcommand{\Tr}{\operatorname{Tr}}
\newcommand{\Pic}{\operatorname{Pic}}

\newcommand{\Spec}{\operatorname{Spec}}
\newcommand{\oo}{\operatornamewithlimits{\otimes}\limits}
\newcommand{\by}[1]{\overset{#1}{\longrightarrow}}
\newcommand{\yb}[1]{\overset{#1}{\longleftarrow}}
\newcommand{\iso}{\by{\sim}}
\newcommand{\osi}{\yb{\sim}}
\newcommand{\rat}{{\operatorname{rat}}}
\newcommand{\eff}{{\operatorname{eff}}}
\newcommand{\Zar}{{\operatorname{Zar}}}
\newcommand{\Nis}{{\operatorname{Nis}}}
\newcommand{\surj}{\rightarrow\!\!\!\!\!\rightarrow}
\newcommand{\Surj}{\relbar\joinrel\surj}
\newcommand{\Res}{\operatorname{Res}}

\newcommand{\id}{{\operatorname{id}}}
\newcommand{\divi}{{\operatorname{div}}}
\newcommand{\Supp}{{\operatorname{Supp}}}

\newcommand{\Gal}{{\operatorname{Gal}}}
\newcommand{\pd}{{\partial}}

\newcommand{\car}{\operatorname{char}}

\renewcommand{\phi}{\varphi}
\renewcommand{\epsilon}{\varepsilon}
\renewcommand{\div}{\operatorname{div}}

\newcounter{spec}
\newenvironment{thlist}{\begin{list}{\rm{(\roman{spec})}}%
{\usecounter{spec}\labelwidth=20pt\itemindent=0pt\labelsep=10pt}}%
{\end{list}}%

\newenvironment{resume}{\noindent{\sc r\'esum\'e.}\ }{}

\swapnumbers
\newtheorem{lemma}{Lemma}[section]
\newtheorem{thm}[lemma]{Theorem}
\newtheorem{theorem}[lemma]{Theorem}
\newtheorem{prop}[lemma]{Proposition}
\newtheorem{proposition}[lemma]{Proposition}
\newtheorem{cor}[lemma]{Corollary}
\newtheorem{corollary}[lemma]{Corollary}
\theoremstyle{definition}
\newtheorem{defn}[lemma]{Definition}
\newtheorem{constr}[lemma]{Construction}
\newtheorem{definition}[lemma]{Definition}
\newtheorem{nota}[lemma]{Notation}
\newtheorem{para}[lemma]{}
\theoremstyle{remark}
\newtheorem{qn}[lemma]{Question}
\newtheorem{rk}[lemma]{Remark}
\newtheorem{remark}[lemma]{Remark}

\newtheorem{example}[lemma]{Example}
\numberwithin{equation}{section}

\begin{document}
\title{Voevodsky's motives and Weil reciprocity}
\author{Bruno Kahn}
\address{Institut de Math\'ematiques de Jussieu\\UMR 7586\\ Case 247\\4 place
Jussieu\\75252 Paris Cedex 05\\France}
\email{kahn@math.jussieu.fr}
\author{Takao Yamazaki}
\address{Institute of Mathematics, Tohoku University, Aoba, Sendai, 980-8578, Japan}
\email{ytakao@math.tohoku.ac.jp}
\date{November 29, 2012}
\thanks{The second author is supported by 
Grant-in-Aid for Challenging Exploratory Research (22654001),
Grand-in-Aid for Young Scientists (A) (22684001),
and Inamori Foundation.}



\begin{abstract}
We describe So\-me\-ka\-wa's $K$-group associated to a finite
collection of semi-abelian varieties (or more general sheaves) in terms of the tensor product in Voevodsky's category of motives. While Somekawa's definition is based on Weil reciprocity, Voevodsky's category is based on homotopy invariance. We apply this to explicit descriptions of certain algebraic cycles.
\end{abstract}
\maketitle

\tableofcontents

\section{Introduction}
\begin{para}In this article, we construct an isomorphism
\begin{equation}\label{eq1}
K(k; \sF_1,\dots,\sF_n)\iso 
\Hom_{\DM_-^\eff}(\Z,\sF_1[0]\otimes\dots\otimes \sF_n[0]).
\end{equation}

Here $k$ is a perfect field, and
$\sF_1,\dots, \sF_n$ are homotopy invariant Nisnevich sheaves with
transfers in the sense of \cite{voetri}.
On the right hand side, 
the tensor product $\sF_1[0]\otimes\dots\otimes \sF_n[0]$ 
is computed in Voevodsky's triangulated category $\DM_-^\eff$
of effective motivic complexes. 
The group $K(k; \sF_1,\dots,\sF_n)$ 
 will be defined in Definition \ref{def:k-groups}
by an explicit set of generators and relations:
it is a generalization of the group
which was defined by K. Kato and studied by M. Somekawa in \cite{somekawa} when $\sF_1,\dots,\sF_n$ are semi-abelian varieties.  
\end{para}

\begin{para}
In the introduction of \cite{somekawa}, 
Somekawa wrote that he expected an isomorphism of the form
\[K(k; G_1,\dots,G_n)\simeq \Ext^n_{MM}(\Z,G_1[-1]\otimes\dots\otimes G_n[-1])\]
where 
$MM$ is a conjectural abelian category of mixed motives over $k$,
$G_1, \dots, G_n$ are semi-abelian varieties over $k$,
and $G_1[-1], \dots, G_n[-1]$ are the corresponding $1$-motives.
Since we do not have such a category $MM$ at hand,
\eqref{eq1} provides the closest approximation to Somekawa's expectation.
\end{para}

%

\begin{para}\label{para:kato} 
The most basic case of  \eqref{eq1} is $\sF_1=\cdots=\sF_n=\G_m$.
By \cite[Theorem 1.4]{somekawa}, the left hand side is  isomorphic to the usual Milnor $K$-group $K_n^M(k)$.
The right hand side  is almost by definition
the motivic cohomology group $H^n(k,\Z(n))$.
Thus,
when  $k$ is perfect,
we get a new and less combinatorial proof
of the Suslin-Voevodsky isomorphism
\cite[Thm. 3.4]{sus-voe}, \cite[Thm. 5.1]{mvw}
\begin{equation}\label{eq:sv}
 K_n^M(k) \simeq H^n(k,\Z(n)).
\end{equation}
\end{para}

\begin{para}\label{s1.4}
The isomorphism \eqref{eq1} also has the following application 
to algebraic cycles. 
Let $X$ be a $k$-scheme of finite type.
Write $\uCH_0(X)$ for
the homotopy invariant Nisnevich sheaf with transfers
\[U\mapsto CH_0(X\times_k {k(U)}) \quad (U \text{ smooth connected)}\]
see \cite[Th. 2.2]{motiftate}.
Let $i, j \in \Z$. 
We write $CH_i(X, j)$ for Bloch's homological higher Chow group \cite[1.1]{levine}: if $X$ is equidimensional of dimension $d$, it agrees with the group $CH^{d-i}(X,j)$ of \cite{bloch}.
\end{para}

\begin{theorem}\label{prop:motivichom}
Suppose $\car k=0$.
Let $X_1, \dots, X_n$ be quasi-projective $k$-schemes.
Put $X=X_1 \times \dots \times X_n$.
For any $r \geq 0$, we have an isomorphism
\begin{align}\label{eq:cycle}
  K(k; \uCH_0(X_1), \dots, \uCH_0(X_n), \G_m, \dots, \G_m)
   \iso CH_{-r}(X, r),
\end{align}
where we put $r$ copies of $\G_m$ on the left hand side.\footnote{
Using recent results of Shane Kelly \cite{kelly}, one may remove the characteristic zero hypothesis if we invert the exponential characteristic of $k$. 
}
\end{theorem}


\begin{para}
When $X_1,\dots,X_n$ are smooth projective,\footnote{
In this case, Theorem \ref{prop:motivichom} is valid 
in any characteristic, see Remark \ref{rem:sm-proj}.
}
special cases of \eqref{eq:cycle} were previously known.
The case $r=0$ was proved by Raskind and Spie\ss\ \cite[Corollary 4.2.6]{rs},
and the case $n=1$ was proved by Akhtar \cite[Theorem 6.1]{akhtar}
(without assuming $k$ to be perfect).
The extension to non smooth projective varieties is new and nontrivial.
\end{para}

\begin{para} Theorem \ref{prop:motivichom} is proven using the Borel-Moore motivic homology introduced in \cite[\S 9]{fv}. We also have a variant which involves motivic homology, see Theorem \ref{prop:higherchow}.
Here is an application. Let $C_1,C_2$ be two smooth connected curves over our perfect field $k$, and put $S=C_1\times C_2$. Assume that $C_1$ and $C_2$ both have a $0$-cycle of degree $1$. Then  the special case $n=2,r=0$ of Theorem \ref{prop:higherchow} gives an isomorphism\[\Z\oplus \Alb_S(k)\oplus K(k;A_1,A_2)\iso H_0(S,\Z)\]
where $A_i$ is the Albanese variety of $C_i$  (compare \cite[Th. 3.4.2]{voetri}), $\Alb_S=A_1\times A_2$ is the Albanese variety of $S$ and the right hand side in this case is Suslin homology \cite{sus-voe2}, see \S \ref{s8.1}. 

Since Somekawa's groups are defined in an explicit manner,
one can sometimes determine the structure of $K(k; A_1, A_2)$ completely.
For instance, when $k$ is finite, we have $K(k;A_1,A_2)=0$ by \cite{knote}.
This immediately implies the bijectivity of the generalized Albanese map
\[a_S:H_0(S,\Z)^{\deg=0}\to \Alb_S(k)\]
of Ramachandran and Spie\ss-Szamuely \cite{spsz}. Note that $a_S$ is not bijective for a smooth projective surface $S$ in general, see \cite[Prop. 9]{ks2}.
\end{para}

\begin{para}
We conclude this introduction
by pointing out the main difficulty and main ideas
in the proof of \eqref{eq1}.

The definitions of the two sides of \eqref{eq1} are
quite different:
the left hand side is based on \emph{Weil reciprocity},
while the right hand side is based on \emph{homotopy invariance}. Thus it is not even obvious how to define a map \eqref{eq1} to start with. Our solution is to write both sides as quotients of a common larger group, and to prove that one quotient factors through the other.
This provides a map \eqref{eq1} which is automatically surjective (Theorem \ref{p1}).

The proof of its injectivity 
turns out to be much more difficult.
We need  to find many relations
coming from Weil reciprocity.
Our main idea,
inspired by \cite[Theorem 1.4]{somekawa}
(recalled in \S \ref{para:kato}),
is to use the \emph{Steinberg relation}
to create Weil reciprocity relations.
To show that this  provides us with
enough such relations,
we need to carry out a heavy computation of symbols in \S \ref{sect:isom}.
\end{para}

\subsection*{Acknowledgements} 
Work in this direction had been done previously by Mo\-chi\-zu\-ki
\cite{mochi}. 
The surjective map \eqref{eq1}
was announced in \cite[Remark 10 (b)]{sp-ya}.
This research was started by the first author, who wrote
the first part of this paper \cite{somprem}.
The collaboration began when the second author visited
the Institute of Mathematics of Jussieu in October 2010.
Somehow, the research accelerated
after the earthquake on March 11, 2011 in Japan.
We wish to acknowledge the pleasure of such a fruitful collaboration, along these
circumstances.

We acknowledge the depth of the ideas of Milnor, Kato, Somekawa, Suslin and Voevodsky.
Especially we are impressed by the relevance of the Steinberg relation in this story.

\section{Mackey functors and presheaves with transfers}\label{s:mack}

\begin{para}A \emph{Mackey functor} over $k$ is a contravariant additive (i.e., commuting with
coproducts) functor $A$ from the category of \'etale $k$-schemes to the category of abelian
groups, provided with a covariant structure verifying the following exchange condition: if 
\[\begin{CD}
Y'@>f'>> Y\\
@Vg'VV @VgVV\\
X'@>f>> X
\end{CD}
\]
is a cartesian square of \'etale $k$-schemes, then the diagram
\[\begin{CD}
A(Y')@>{f'}^*>> A(Y)\\
@Vg'_*VV @Vg_*VV\\
A(X')@>f^*>> A(X)
\end{CD}
\]
commutes. Here, $^*$ denotes the contravariant structure while $_*$ denotes the covariant
structure. The Mackey functor $A$ is \emph{cohomological} if we further have
\[f_* f^* = \deg(f)\]
for any $f:X'\to X$, with $X$ connected. We denote by $\Mack$ the abelian category of Mackey
functors, and by $\Mack_c$ its full subcategory of cohomological Mackey functors.
\end{para}

\begin{para}Classically \cite[(1.4)]{thevenaz}, a Mackey functor may be viewed as a
contravariant additive functor on the category $\Span$ of ``spans" on \'etale $k$-schemes,
defined as follows: objects are \'etale $k$-schemes. A morphism from $X$ to $Y$ is an
equivalence class of diagram (span)
\begin{equation}\label{eq3}
X\yb{g} Z\by{f} Y.
\end{equation}

Composition of spans is defined via fibre product in an obvious manner. If $A$ is a Mackey
functor, the corresponding functor on $\Span$ has the same value on objects, while its value on
a span \eqref{eq3} is given by $g_*f^*$.

Note that $\Span$ is a preadditive category: one may add (but not subtract) two morphisms
with same source and target. We may as well view a Mackey functor as an additive functor on the
associated additive category $\Z\Span$.
\end{para}

\begin{para}Let $\Cor$ be Voevodsky's category of finite correspondences on smooth
$k$-schemes, denoted by $SmCor(k)$ in \cite[\S 2.1]{voetri}. The category $\Z\Span$ is
isomorphic to its full subcategory consisting of smooth $k$-schemes of dimension $0$
(= \'etale $k$-schemes). In particular, any presheaf with transfers in the sense of Voevodsky
\cite[Def. 3.1.1]{voetri} restricts to a Mackey functor over $k$. By \cite[Cor. 3.15]{voepre},
the restriction of a \emph{homotopy invariant} presheaf with transfers yields a cohomological
Mackey functor. In other words, we have exact functors
\begin{align}
\rho:\PST&\to \Mack\label{eq4}\\
\rho:\HI&\to \Mack_c\label{eq5}
\end{align}
where $\PST$ denotes the category of presheaves with transfers (contravariant additive
functors from $\Cor$ to abelian groups) and $\HI$ is its full subcategory consisting of
homotopy invariant presheaves with transfers.
\end{para}

\begin{para} There is a tensor product of Mackey functors $\oo^M$, originally defined by L. G.
Lewis (unpublished): it extends naturally the symmetric
monoidal structure $(X,Y)\mapsto X\times_K Y$ on $\Z\Span$ via the additive Yoneda embedding
(see \S \ref{A!}). If either $A$ or $B$ is cohomological, $A\oo^M B$ is cohomological. 
This tensor product is the same as the one defined in \cite[\S 5]{decomposable} and
\cite{knote}: this follows from
\eqref{eqA.2} and the fact that $\Z\Span$ is rigid, all objects being self-dual
(indeed, $\Z\Span$ is canonically isomorphic to the category of Artin Chow motives with
integral coefficients).

For the reader's convenience, we recall the definition of $\oo^M$.
Let $A_1, \dots, A_n$ be Mackey functors.
For any \'etale $k$-scheme $X$,
we define 
\[ (A_1\oo^M\dots \oo^M A_n)(X) :=
[\bigoplus_{Y \to X} A_1(Y) \otimes \dots \otimes A_n(Y)]/R,
\] 
where $Y \to X$ runs through all finite \'etale morphisms,
and $R$ is the subgroup generated by all elements 
of the form
\[ a_1 \otimes \dots \otimes f_*(a_i) \otimes \dots a_n
 -f^*(a_1) \otimes \dots \otimes a_i \otimes \dots f^*(a_n),
\]
where $Y_1 \overset{f}{\to} Y_2 \to Y$ is a tower of \'etale morphisms,
$1 \leq i \leq n$,
$a_i \in A_i(Y_1)$ and $a_j \in A_j(Y_2) ~(j=1, \dots, i-1, i+1, \dots, n)$.
\end{para}

\begin{para}
There is a tensor product on presheaves with transfers defined exactly in the same way \cite[p.
206]{voetri}.  
\end{para}

\begin{para} By definition, the functor \eqref{eq4} equals $i^*$, where $i$ is the inclusion
$\Z\Span\to \Cor$. This inclusion has a left adjoint $\pi_0$ (scheme of constants). Both
functors
$i$ and
$\pi_0$ are symmetric monoidal: for $\pi_0$, reduce to the case where $k$ is separably closed.
\end{para}

\begin{para}\label{s2.5}  By \S \S \ref{Aadj} and \ref{Acoh}, this implies that \eqref{eq4} is
symmetric monoidal. In other words, if $\sF$ and $\sG$ are presheaves with transfers, then
\begin{equation}\label{eq6}
\rho\sF\oo^M\rho\sG\simeq \rho(\sF\otimes_{\PST}\sG).
\end{equation}
The left hand side is sometimes abbreviated to $\sF \oo^M \sG$.
\end{para}

\begin{para}\label{s2.6} The inclusion functor $\HI\to \PST$ has a left adjoint $h_0$, and the
symmetric  monoidal structure of $\PST$ induces one on $\HI$ via $h_0$. In other words, if
$\sF,\sG\in \HI$, we define
\begin{equation}\label{eq8}
\sF\otimes_{\HI}\sG = h_0(\sF\otimes_{\PST} \sG).
\end{equation}

Note that \eqref{eq5} is \emph{not} symmetric monoidal (since it is the restriction of
\eqref{eq4}).
\end{para}

\begin{para}\label{ssurj} For any $\sF\in \PST$, the unit morphism $\sF\to h_0(\sF)$ induces a
surjection
\begin{equation}\label{eq:h0}
\sF(k) \surj h_0(\sF)(k).
\end{equation}
This is obvious from the formula $h_0(\sF)=\Coker(C_1(\sF)\to \sF)$.
\end{para}

\begin{para}\label{s2.8} We shall also need to work with Nisnevich sheaves with transfers. We
denote by $\NST$ the category of Nisnevich sheaves with transfers (objects of $\PST$ which are
sheaves in the Nisnevich topology). 
By \cite[Theorem 3.1.4]{voetri}, the inclusion functor $\NST\to
\PST$ has an exact left adjoint $\sF\mapsto \sF_\Nis$ (sheafification). The category $\NST$
then inherits a tensor product by the formula
\[\sF\otimes_{\NST} \sG=(\sF\otimes_{\PST}\sG)_\Nis.\]

Similarly, we define $\HI_\Nis=\HI\cap \NST$. The sheafification functor restricts to an exact
functor $\HI\to \HI_\Nis$ \cite[Theorem 3.1.11]{voetri}, and $\HI_\Nis$ gets a tensor product by the
formula
\[\sF\otimes_{\HI_\Nis} \sG=(\sF\otimes_{\HI}\sG)_\Nis.\]

To summarize, all functors in the following naturally commutative diagram are symmetric
monoidal:
\begin{equation}\label{eq15}
\begin{CD}
\PST@>\Nis>> \NST\\
@V{h_0}VV @V{h_0^\Nis}VV\\
\HI@>\Nis>> \HI_\Nis.
\end{CD}
\end{equation}
where each functor is left adjoint to the corresponding inclusion.
\end{para}

\begin{para}\label{s3.3} Let $\sF$ be a presheaf on $Sm/k$, and let $\sF_\Nis$ be the
associated Nisnevich sheaf. Then we have an isomorphism
\begin{equation}\label{eq11}
\sF(k)\iso \sF_\Nis(k).
\end{equation}

Indeed, any covering of $\Spec k$ for the Nisnevich topology refines to a trivial covering.
In particular, the functor $\sF\mapsto \sF_\Nis(k)$ is exact.

This applies in particular to a presheaf with transfers and the associated Nisnevich sheaf with
transfers.
\end{para}

\begin{para}\label{mor1} 
Let $\sF_1, \dots, \sF_n \in \HI_\Nis$.
Then \eqref{eq6} yields a canonical isomorphism
\begin{equation}\label{eq7}
(\sF_1\oo^M\dots \oo^M \sF_n)(k)\simeq 
(\sF_1\otimes_{\PST}\dots \otimes_{\PST} \sF_n)(k).
\end{equation}
Composing
\eqref{eq7} with the unit morphism 
$Id\Rightarrow h_0^\Nis$ from \eqref{eq15} and taking
\eqref{eq8} into account, we get a canonical morphism
\begin{equation}\label{eq9}
(\sF_1\oo^M\dots \oo^M \sF_n)(k)
\to (\sF_1\otimes_{\HI_\Nis}\dots \otimes_{\HI_\Nis}\sF_n)(k).
\end{equation}
which is surjective by \S\S \ref{ssurj} and \ref{s3.3}.
\end{para}

\begin{para}\label{fullalb}
If $G$ is a commutative $k$-group scheme whose identity component is a
quasi-projective variety, then $G$ has a canonical structure of 
Nisnevich sheaf with transfers
(\cite[proof of Lemma 3.2]{spsz} completed by \cite[Lemma 1.3.2]{bar-kahn}). 
This applies in particular to semi-abelian varieties
and also to the "full" Albanese scheme \cite{ram} of a smooth variety
(which is an extension of a lattice by a a semi-abelian variety). 
In particular, if $G_1,\dots,G_n$ are such $k$-group schemes,
\eqref{eq9} yields a canonical surjection
\begin{equation}\label{eq9-semiabel}
(G_1\oo^M\dots \oo^MG_n)(k)\to (G_1\otimes_{\HI_\Nis}\dots \otimes_{\HI_\Nis}G_n)(k),
\end{equation}
where the $G_i$ are considered on the left as Mackey functors, 
and on the right as homotopy invariant Nisnevich sheaves with transfers. 
\end{para}

\section{Presheaves with transfers and motives}

\begin{para} The left adjoint $h_0^\Nis$ in \eqref{eq15} ``extends" to a left adjoint $C_*$ of
the inclusion
\[\DM_-^\eff\to D^-(\NST)\]
where the left hand side is Voevodsky's triangulated category of effective motivic complexes
\cite[\S 3, esp. Prop. 3.2.3]{voetri}. 

More precisely, $\DM_-^\eff$ is defined as the full subcategory of objects of $D^-(\NST)$
whose cohomology sheaves are homotopy invariant. The canonical $t$-structure of $D^-(\NST)$
induces a $t$-structure on $\DM_-^\eff$, with heart $\HI_\Nis$. The functor $C_*$ is right
exact with respect to these $t$-structures, and if $\sF\in \NST$, then
$H_0(C_*(\sF))=h_0^\Nis(\sF)$.
\end{para}

\begin{para} The tensor structure of \S \ref{s2.8} on $\NST$ extends to one on $D^-(\NST)$
\cite[p. 206]{voetri}. Via
$C_*$, this tensor structure descends to a tensor structure on $\DM_-^\eff$ \cite[p.
210]{voetri}, which will simply be denoted by
$\otimes$. The relationship between this tensor structure and the one of \S \ref{s2.8} is as
follows: if $\sF,\sG\in \HI_\Nis$, then
\begin{equation}\label{eq10}
\sF\otimes_{\HI_\Nis}\sG = H^0(\sF[0]\otimes \sG[0])
\end{equation}
where $\sF[0],\sG[0]$ are viewed as complexes of Nisnevich sheaves with transfers concentrated
in degree $0$.

We shall need the following lemma, which is not explicit in \cite{voetri}:
\end{para}

\begin{lemma}\label{l2} The tensor product $\otimes$ of $\DM_-^\eff$ is right exact with respect to the homotopy $t$-structure.
\end{lemma}

\begin{proof} By definition,
\[C\otimes D = C_*(C\oo^L D)\]
for $C,D\in \DM_-^\eff$, where $\oo^L$ is the tensor product of $D^-(\NST)$ defined in
\cite[p. 206]{voetri}. We want to show that, if $C$ and $D$ are concentrated in degrees $\le
0$, then so is $C\otimes D$. Using the canonical left resolutions of loc. cit., it is enough to
do it for $C$ and $D$ of the form $C_*(L(X))$ and $C_*(L(Y))$ for two smooth schemes $X,Y$.
Since $C_*$ is symmetric monoidal, we have
\[C_*(L(X))\otimes C_*(L(Y))\osi C_*(L(X)\oo^L L(Y)) = C_*(L(X\times Y))\]
and the claim is obvious in view of the formula for $C_*$ \cite[p. 207]{voetri}.
\end{proof}

\begin{para} Let $C\in \DM_-^\eff$. For any $X\in Sm/k$ and any $i\in\Z$, we have
\[\bH^i_\Nis(X,C)\simeq \Hom_{\DM_-^\eff}(M(X),C[i])\]
where $M(X)=C_*(L(X))$ is the motive of $X$ computed in $\DM_-^\eff$ (cf. \cite[Prop.
3.2.7]{voetri}).

Specializing to the case $X=\Spec k$ ($M(X)=\Z$) and taking \S \ref{s3.3} into account, we get
\begin{equation}\label{eq12}
\Hom_{\DM_-^\eff}(\Z,C[i])\simeq H^i(C)(k).
\end{equation}

Combining \eqref{eq10}, \eqref{eq11} and \eqref{eq12}, we get:
\end{para}

\begin{lemma} \label{l1} Let $\sF_1,\dots,\sF_n$ be homotopy invariant Nisnevich sheaves with
transfers. Then we have a canonical isomorphism
\begin{equation}\label{eq:dm}
(\sF_1\otimes_{\HI_\Nis}\dots \otimes_{\HI_\Nis}\sF_n)(k) \simeq
\Hom_{\DM_-^\eff}(\Z,\sF_1[0]\otimes\dots \otimes\sF_n[0]).
\end{equation}
\end{lemma}

\begin{para}\label{mor2} 
Summarizing, 
for any $\sF_1,\dots,\sF_n \in \HI_\Nis$ 
we get a surjective homomorphism 
\begin{equation}\label{eq2}
(\sF_1\oo^M\dots \oo^M \sF_n)(k)\to \Hom_{\DM_-^\eff}(\Z,\sF_1[0]\otimes\dots\otimes \sF_n[0]).
\end{equation}
by composing
\eqref{eq7}, 
\eqref{eq:h0}, 
\eqref{eq8}, 
\eqref{eq11}
and \eqref{eq:dm}.

\end{para}

\section{Presheaves with transfers and local symbols}

\begin{para}\label{s4.2} Given a presheaf with transfers $\sG$, recall from \cite[p.
96]{voepre} the pre\-sheaf with transfers $\sG_{-1}$ defined by the formula
\begin{equation}\label{eq:defofminusone}
\sG_{-1}(U) = \Coker\left(\sG(U\times \A^1)\to \sG(U\times (\A^1-\{0\}))\right).
\end{equation}

Suppose that $\sG$ is homotopy invariant. Let $X\in Sm/k$ (connected),  $K=k(X)$ and $x\in X$ be
a point of codimension $1$. By \cite[Lemma 4.36]{voepre}, there is a canonical isomorphism
\begin{equation}\label{eq14a}
\sG_{-1}(k(x))\simeq H^1_x(X,\sG_\Zar)
\end{equation}
yielding a
canonical map
\begin{equation}\label{eq14}
\partial_x:\sG(K)\to \sG_{-1}(k(x)).
\end{equation}
\end{para}

The following lemma follows from the construction of the isomorphisms \eqref{eq14a}. It is part
of the general fact that $\sG$ defines a cycle module in the sense of Rost (cf. \cite[Prop.
5.4.64]{deglise}).

\begin{lemma}\label{l4} a) Let $f:Y\to X$ be a dominant morphism, with
$Y$ smooth and connected. Let $L=k(Y)$, and let $y\in Y^{(1)}$ be such that $f(y)=x$. Then the
diagram
\[\begin{CD}
\sG(L)@>(\partial_y)>> \sG_{-1}(k(y))\\
@Af^*AA @Ae f^*AA\\
\sG(K)@>\partial_x>> \sG_{-1}(k(x))
\end{CD}\]
commutes, where $e$ is the ramification index of $v_y$ relative to $v_x$.\\
b) If $f$ is finite surjective, the diagram
\[\begin{CD}
\sG(L)@>(\partial_y)>> \displaystyle\bigoplus_{y\in f^{-1}(x)} \sG_{-1}(k(y))\\
@Vf_*VV @Vf_*VV\\
\sG(K)@>\partial_x>> \sG_{-1}(k(x))
\end{CD}\]
commutes.\qed
\end{lemma}

\begin{prop}\label{p2} Let $\sG\in \HI_\Nis$. There is a canonical isomorphism
\[\sG_{-1} = \uHom(\G_m,\sG).\]
\end{prop}

\begin{proof} This may not be the most economic proof, but it is quite short. The statement
means that $\sG_{-1}$ represents the functor
\[\sH\mapsto \Hom_{\HI_\Nis}(\sH\otimes_{\HI_\Nis}\G_m,\sG).\]

By \cite[Lemma 4.35]{voepre}, we have
\[\sG_{-1} = \Coker(\sG\to p_*p^*\sG)\]
where $p:\A^1-\{0\}\to \Spec k$ is the structural morphism and $p_*, p^*$ are computed with
respect to the Zariski topology. By \cite[Theorem 5.7]{voepre}, we may replace the Zariski topology
by the Nisnevich topology. Moreover, by \cite[Prop. 5.4 and Prop. 4.20]{voepre}, we have
$R^ip_*p^*\sG = 0$ for $i>0$, hence $p_*p^*\sG[0]\iso Rp_*p^*\sG[0]$.

By \cite[Prop. 3.2.8]{voetri}, we have
\[Rp_*p^* \sG[0] = \uHom(M(\A^1-\{0\}),\sG[0])\]
where $\uHom$ is the (partially defined) internal Hom of $\DM_-^\eff$. By \cite[Prop.
3.5.4]{voetri} (Gysin triangle) and homotopy invariance, we have an exact triangle, split by
any rational point of $\A^1-\{0\}$:
\[\Z(1)[1]\to M(\A^1-\{0\})\to \Z\by{+1}\]

To get a canonical splitting, we may choose the rational point $1\in \A^1-\{0\}$.

By \cite[Cor. 3.4.3]{voetri}, we have an isomorphism $\Z(1)[1]\simeq \G_m[0]$. Hence, in
$\DM_-^\eff$, we have an isomorphism
\[\sG_{-1}[0]\simeq \uHom(\G_m[0],\sG[0]).\]

Let $\sH\in \HI_\Nis$. We get:
\begin{multline*}
\Hom_{\DM_-^\eff}(\sH[0],\sG_{-1}[0]) \simeq \Hom_{\DM_-^\eff}(\sH[0]\otimes \G_m[0],\sG[0])\\
\simeq \Hom_{\HI_\Nis}(H^0(\sH[0]\otimes \G_m[0]),\sG)=:\Hom_{\HI_\Nis}(\sH\otimes_{\HI_\Nis}
\G_m,\sG)
\end{multline*}
as desired (see \eqref{eq10}). For the second isomorphism, we have used the right exactness of
$\otimes$ (Lemma
\ref{l2}).
\end{proof}

\begin{rk}\label{r1} The proof of Proposition \ref{p2} also shows that, in $\DM_-^\eff$, we
have an isomorphism
\[\uHom(\G_m[0],\sG[0])\simeq \uHom(\G_m,\sG)[0]\]
where the left $\uHom$ is computed in $\DM_-^\eff$ and the right $\uHom$  is computed in
$\HI_\Nis$. In particular, $\uHom(\G_m[0],-):\DM_-^\eff\to \DM_-^\eff$ is $t$-exact.
\end{rk}

\begin{prop}\label{p4} Let $C$ be a smooth, proper, connected curve over $k$,
with function field $K$. There exists a canonical homomorphism
\[\Tr_{C/k}:H^1_\Zar(C,\sG)\to \sG_{-1}(k)\]
such that, for any $x\in C$, the composition
\[\sG_{-1}(k(x))\simeq H^1_x(C,\sG)\to H^1_\Zar(C,\sG)\by{\Tr_C} \sG_{-1}(k)\]
equals the transfer map $\Tr_{k(x)/k}$ associated to the finite surjective morphism $\Spec
k(x)\to\Spec k$.
\end{prop}

\begin{proof} By  
\cite[Prop. 3.2.7]{voetri}, we have
\[H^1_\Zar(C,\sG)\iso H^1_\Nis(C,\sG)\simeq
\Hom_{\DM_-^\eff}(M(C),\sG[1]).\] 

The structural morphism $C\to\Spec k$ yields a morphism of motives $M(C)\to \Z$ which, by
Poincar\'e duality, yields a canonical morphism
\[\G_m[1]\simeq \Z(1)[2]\to M(C).\]

(One may view this morphism as the image of the canonical morphism $\L\to h(C)$ in the category
of Chow motives.)

Therefore, by Proposition \ref{p2} and Remark \ref{r1}, we get a map
\[\Tr_{C/k}:H^1_\Zar(X,\sG)\to\Hom_{\DM_-^\eff}(\G_m[1],\sG[1])=\sG_{-1}(k).\]

It remains to prove the claimed compatibility. Let $M^x(C)$ be the motive of $C$ with supports
in $x$, defined as $C_*(\Coker(L(C-\{x\})\to L(C))$. Let $\Z_{k(x)}=M(\Spec k(x))$. By
\cite[proof of Prop. 3.5.4]{voetri}, we have an isomorphism $M^x(C)\simeq \Z_{k(x)}(1)[2]$, and
we have to show that the composition
\[
\Z(1)[2]\to M(C)\by{g_x} \Z_{k(x)}(1)[2]
\]
is $\Tr_{k(x)/k}$, up to twisting and shifting.  To see this, we observe that $g_x$ is the image of the morphism of Chow motives
\[h(C)\to h(\Spec k(x))(1)\]
dual to the morphism $h(\Spec k(x))\to h(C)$ induced by the inclusion $\Spec k(x)\to C$: this
is easy to check from the definition of $g_x$ in \cite{voetri} (observe that in this special
case, $Bl_x(C)=C$ and that we may use a variant of the said construction replacing $C\times
\A^1$ by $C\times\P^1$ to stay within smooth projective varieties). The conclusion now follows
from the fact that the composition 
\[\Spec k(x)\to C\to \Spec k\]
is the structural morphism of $\Spec k(x)$.
\end{proof}

\begin{prop}[Reciprocity] \label{p3} Let $C$ be a smooth, proper, connected curve over $k$, with
function field $K$. Then the sequence
\[\begin{CD}
\sG(K)@>{(\partial_x)}>> \bigoplus_{x\in C} \sG_{-1}(k(x))@>\sum_x \Tr_{k(x)/k}>>
\sG_{-1}(k)
\end{CD}\]
is a complex.
\end{prop}

\begin{proof} This follows from Proposition \ref{p4}, since the composition
\[\sG(K)\to \bigoplus_{x\in C} H^1_x(C,\sG)\by{(g_x)} H^1(C,\sG)\]
is $0$.
\end{proof}

\begin{para}\label{s4.4}
If $\sF,\sG$ are presheaves with transfers, there is a bilinear morphism of presheaves with
transfers (i.e. a natural transformation over $\PST\times \PST$):
\begin{multline*}
\sF(U)\otimes \sG_{-1}(V) =\\
\Coker\left(\sF(U)\otimes \sG(V\times\A^1)\to \sF(U)\otimes
\sG(V\times(\A^1-\{0\}))\right)\to\\
 \Coker\left((\sF\otimes_{\PST} \sG)(U\times V\times \A^1)\to (\sF\otimes_{\PST}
\sG)(U\times V\times(\A^1-\{0\}))\right)\\
= (\sF\otimes_{\PST}\sG)_{-1}(U\times V)
\end{multline*}
which induces a morphism
\begin{equation}\label{eq16}
\sF\otimes_{\PST}\sG_{-1}\to (\sF\otimes_{\PST}\sG)_{-1}.
\end{equation}

In particular, for $\sG=\G_m$, we get a morphism $\sF\to (\sF\otimes_{\PST} \G_m)_{-1}$.
\end{para}

\begin{thm}\label{t1} Suppose $\sF\in \HI_\Nis$. Then\\
a) The composition 
\[\sF\to (\sF\otimes_{\PST} \G_m)_{-1}\to (\sF\otimes_{\HI_\Nis} \G_m)_{-1}\]
is the unit map of the adjunction between $-\otimes_{\HI_\Nis} \G_m$ and $(-)_{-1}$
stemming from Proposition \ref{p2}.\\
b) This composition is
an isomorphism.
\end{thm}

\begin{proof} a) is an easy bookkeeping. For b),  we compute again in $\DM_-^\eff$. By
Proposition
\ref{p2}, we are considering the  morphism in $\HI_\Nis$
\begin{equation}\label{eq13}
\sF\to \uHom(\G_m,\sF\otimes_{\HI_\Nis} \G_m).
\end{equation}

Consider the corresponding morphism in $\DM_-^\eff$
\[\sF[0]\to \uHom(\G_m[0],\sF[0]\otimes \G_m[0]).\]

As recalled in the proof of Proposition \ref{p2}, we have $\G_m[0]=\Z(1)[1]$, hence the above
morphism amounts to
\[\sF[0]\to \uHom(\Z(1),\sF[0](1))\]
which is an isomorphism by the cancellation theorem \cite{voecan}. A fortiori, \eqref{eq13},
which is (by Remark \ref{r1}) the $H^0$ of this isomorphism, is an isomorphism.
 \end{proof}

\begin{nota}\label{n1} Let $\sF,\sG\in \HI_\Nis$ and
$\sH=\sF\otimes_{\HI_\Nis}
\sG$. Let
$X,K,x$ be as in \S \ref{s4.2}. For $(a,b)\in \sF(K)\times \sG(K)$,  we denote by
$a\cdot b$ the image of $a\otimes b$ in $\sH(K)$ by the map
\[\sF(K)\otimes \sG(K)\to \sH(K).\]
We define the {\it local symbol} on $\sF$
\[ \sF(K) \times K^* \to \sF(k(x)) 
\]
to be the composition
\[ \sF(K) \times K^* \overset{\cdot}{\to} 
  (\sF\otimes_{\HI_\Nis}\G_m)(K)
  \overset{\pd_x}{\to}
  (\sF\otimes_{\HI_\Nis}\G_m)_{-1}(k(x)) \cong \sF(k(x))
\]
where 
the first map is given by the above construction with $\sG=\G_m$,
and the last isomorphism is given by Theorem \ref{t1}.
The image of $(a, b) \in \sF(K) \times K^*$
by the local symbol is denoted by $\pd_x(a, b) \in \sF(k(x))$.
\end{nota}

\begin{prop}[cf. \protect{\cite[Prop. 5.5.27]{deglise}}]\label{l3} Let $\sF,\sG\in \HI_\Nis$,
and consider the morphism induced by \eqref{eq16}
\[\sF\otimes_{\HI_\Nis} \sG_{-1}\by{\Phi} (\sF\otimes_{\HI_\Nis} \sG)_{-1}.\]
 Let $X,K,x$ be as in \S \ref{s4.2}. Then the diagram
\[\xymatrix{
\sF(\sO_{X,x})\otimes \sG(K)\ar[r]\ar[d]^{i_x^*\otimes \partial_x}&
(\sF\otimes_{\HI_\Nis}\sG)(K)\ar[dd]^{\partial_x}\\
\sF(k(x))\otimes \sG_{-1}(k(x))\ar[d]\\
(\sF\otimes_{\HI_\Nis} \sG_{-1})(k(x))\ar[r]^{\Phi}& (\sF\otimes_{\HI_\Nis} \sG)_{-1}(k(x))
}\]
commutes, where $i_x^*$ is induced by the reduction map $\sO_{X,x}\to k(x)$. In other words,
with Notation \ref{n1} we have the identity
\begin{equation}\label{eq17}
\partial_x(a\cdot b) = \Phi(i_x^*a\cdot \partial_x b)
\end{equation}
for $(a,b)\in \sF(\sO_{X,x})\times \sG(K)$.
\end{prop}

\begin{cor}\label{c1} Let $\sF\in \HI_\Nis$; let $X,K,x$ be as in \S \ref{s4.2} and let
$(a,f)\in \sF(K)\times K^*$. \\
a) Suppose that there is $\tilde{a} \in \sF(\sO_{X,x})$
whose image in $\sF(K)$ is $a$. Then
we have
\[\partial_x(a, f) = v_x(f) a(x)\]
where $a(x)$ is the image of 
$\tilde{a}$ in $\sF(k(x))$
(which is independent of the choice of $\tilde{a}$).\\
b) Suppose that $v_x(f- 1)>0$. Then $\partial_x(a, f) = 0$.
\end{cor}

\begin{proof} a) This follows from Proposition \ref{l3} (applied with $\sG=\G_m$) and Theorem
\ref{t1}. b) This follows again from Proposition \ref{l3}, after switching the r\^oles of $\sF$
and $\sG$.
\end{proof}

\begin{prop}\label{l5} 
Let $G$ be a semi-abelian variety.
The local symbol on $G$ defined in Notation \ref{n1} 
agrees with 
Somekawa's  local symbol \cite[(1.1)]{somekawa} (generalising the
Rosenlicht-Serre local symbol) on $G$.
\end{prop}

\begin{proof} Up to base-changing from $k$ to $\bar k$ (see Lemma \ref{l4} a)), we may assume
$k$ algebraically closed. 
By \cite[Ch. III, Prop. 1]{gacl},
it suffices to show that the local symbol in Notation \ref{n1} 
satisfies the properties in
\cite[Ch. III, Def. 2]{gacl} which characterize 
the Rosenlicht-Serre local symbol.
In this definition, Condition i) is obvious, Condition ii) is
Corollary \ref{c1} b), Condition iii) is Corollary \ref{c1} a) and Condition iv) is Proposition
\ref{p3}. 
\end{proof}

\section{$K$-groups of Somekawa type}\label{s:som}

\begin{definition}\label{def:k-groups}
Let $\sF_1, \dots, \sF_n \in \HI_{\Nis}$.
\\
a)
A {\it relation datum of Somekawa type} for $\sF_1, \dots, \sF_n$
is a collection $(C, h, (g_i)_{i=1, \dots, n})$
of the following objects:
(i) a smooth proper connected curve $C$ over $k$,
(ii) $h \in k(C)^*$,
and
(iii) 
$g_i\in \sF_i(k(C))$ for each $i \in \{1, \dots, n\}$;
which satisfies the condition
\begin{equation}
\label{eq:con-somekawa}
\text{for any $c\in C$, there is $i(c)$ 
such that $c \in R_i$ for all $i\ne i(c)$,}
\end{equation}
where
$R_i := \{ c \in C \mid g_i \in \IM(\sF_i(\sO_{C, c}) \to \sF_i(k(C)))\}$.
\\
b)
We define the \emph{$K$-group of Somekawa type}
$K(k; \sF_1, \dots, \sF_n)$ 
to be the quotient
of 
$(\sF_1 \oo^M \dots \oo^M \sF_n)(k)$
by its subgroup generated by elements of the form
\begin{equation}\label{eq:rel-somekawa}
\sum_{c\in C} \Tr_{k(c)/k}(g_1(c)\otimes \dots \otimes
\partial_c(g_{i(c)},h)\otimes\dots\otimes g_n(c)) 
\end{equation}
where $(C, h, (g_i)_{i=1, \dots, n})$ runs through all 
relation data of Somekawa type.
\end{definition}

\begin{remark}
In view of Proposition \ref{l5},
our group $K(k; \sF_1, \dots \sF_n)$
coincides with the Milnor $K$-group defined in \cite{somekawa}
when $\sF_1, \dots, \sF_n$ are semi-abelian varieties over $k$.\footnote{As
was observed by W. Raskind, the signs appearing in
\protect\cite[(1.2.2)]{somekawa} should not be there
(cf. \cite[p. 10, footnote]{rs}).}
\end{remark}

\begin{thm}\label{p1} 
Let $\sF_1, \dots, \sF_n \in \HI_{\Nis}$.
The homomorphism \eqref{eq9} factors through 
$K(k; \sF_1, \dots, \sF_n)$.
Consequently, we get a surjective homomorphism \eqref{eq1}.
\end{thm}
\begin{proof} 
Put $\sF := \sF_1\otimes_{\HI_\Nis} \dots \otimes_{\HI_\Nis} \sF_n$.
Let $(C, h, (g_i)_{i=1, \dots, n})$ be a relation datum of Somekawa type.
We must show that the element
\eqref{eq:rel-somekawa} goes to $0$ in $\sF(k)$ via \eqref{eq9}.
Consider the element $g=g_1\otimes\dots\otimes g_n\in \sF(K)$. 
It follows from  \eqref{eq17} that, 
for any $c\in C$, we have
\begin{multline*}
g_1(c)\otimes \dots \otimes \partial_c(g_{i(c)},h)\otimes\dots\otimes g_n(c)\\
=g_1(c)\otimes \dots \otimes \partial_c(g_{i(c)}\otimes\{h\})\otimes\dots\otimes g_n(c)=
\partial_c(g\otimes \{h\}).
\end{multline*}
The claim now follows from Proposition \ref{p3}.
\end{proof}

\section{$K$-groups of geometric type}\label{sect:6}

\begin{definition}\label{def:k'-groups}
Let $\sF_1, \dots, \sF_n \in \PST$.
\\
a)
A {\it relation datum of geometric type} for $\sF_1, \dots, \sF_n$
is a collection $(C, f, (g_i)_{i=1, \dots, n})$
of the following objects:
(i) a smooth projective connected curve $C$ over $k$,
(ii) a surjective morphism $f: C \to \mathbf{P}^1$, 
(iii) $g_i\in \sF_i(C')$ for each $i \in \{1, \dots, n\}$,
where $C'=f^{-1}(\mathbf{P}^1 \setminus \{ 1 \})$.
\\
b)
We define the \emph{$K$-group of geometric type}
$K'(k; \sF_1, \dots, \sF_n)$ 
to be the quotient
of 
$(\sF_1 \oo^M \dots \oo^M \sF_n)(k)$
by its subgroup generated by elements of the form
\begin{equation}\label{eq:rel-voevod}
  \sum_{c\in C'} v_c(f) \Tr_{k(c)/k}
  (g_1(c) \otimes \dots\otimes g_n(c)) 
\end{equation}
where $(C, f, (g_i)_{i=1, \dots, n})$ runs through all 
relation data of geometric type.
(Here we used the notation $g_i(c) = \iota_c^*(g_i) \in \sF(k(c))$,
where $\iota_c : c = \Spec k(c) \to C'$ is the closed immersion.)
\end{definition}

The rest of this section is devoted to a proof of
the following theorem:
\begin{theorem}\label{thm:k'-group}
Let $\sF_1, \dots, \sF_n \in \HI_{\Nis}$.
The homomorphism \eqref{eq9} induces an isomorphism
\begin{equation}\label{eq1'}
K'(k; \sF_1,\dots,\sF_n)\iso
\Hom_{\DM_-^\eff}(\Z,\sF_1[0]\otimes\dots\otimes \sF_n[0]).
\end{equation}
\end{theorem}

\begin{para}\label{sect:6-1}
For a smooth variety $X$ over $k$,
denote as usual by $L(X)$ the Nisnevich sheaf with transfers
represented by $X$.
Recall that $L(X)(U) = c(U, X)$ 
is the group of finite correspondences
for any smooth variety $U$ over $k$,
viz. the free abelian group on
the set of closed integral subschemes of 
$U \times X$
which are finite and surjective over some irreducible component of $U$.
A morphism $X \to X'$ of smooth varieties
induces a map $L(X) \to L(X')$ of 
Nisnevich sheaves with transfers.

We recall two facts from \cite[p. 206]{voetri},
which are fundamental in the definition of
the tensor product in $\PST$.
\begin{enumerate}
\item
For any $\sF \in \PST$,
there is a surjective map
$\oplus_X L(X) \to \sF$ 
of presheaves with transfers,
where $X$ runs through a (huge) set of smooth varieties over $k$.
\item
We have (by definition)
$L(X) \otimes_{\PST} L(Y) = L(X \times Y)$
for smooth varieties $X$ and $Y$.
\end{enumerate}
\end{para}

\begin{para}\label{sect:6-2}
Let $\sF \in \PST$.
Suppose that we are given the following data:
(i) a smooth projective connected curve $C$ over $k$,
(ii) a surjective morphism $f: C \to \mathbf{P}^1$, 
(iii) a map $\alpha : L(C') \to \sF$ in $\PST$,
where $C'=f^{-1}(\Delta)$ and
$\Delta = \P^1 \setminus \{ 1 \} (\cong \A^1)$.
To such a triple $(C, f, \alpha)$,
we associate an element
\begin{equation}\label{eq:relation}
\alpha( \div(f) ) \in \sF(k),
\end{equation}
where we regard 
$\div(f)$ as an element of 
$Z_0(C') = c(\Spec k, C') = L(C')(k).$

One can rewrite the element \eqref{eq:relation} as follows.
The map $\alpha : L(C') \to \sF$ 
can be regarded as a section $\alpha \in \sF(C')$.
To each closed point $c \in C'$,
we write $\alpha(c)$ for the image of $\alpha$ in $\sF(k(c))$
by the map induced by $c = \Spec k(c) \to C'$.
Now \eqref{eq:relation} is rewritten as
\begin{equation}\label{eq:relation'}
  \sum_{c \in C'} v_c(f) \Tr_{k(c)/k} \alpha(c). 
\end{equation}
\end{para}

\begin{proposition}\label{prop:6-1}
Let $\sF \in \PST$.
We define $\sF(k)_{\rat}$ to be the subgroup of $\sF(k)$
generated by elements \eqref{eq:relation} 
for all triples $(C, f, \alpha)$
as in \S \ref{sect:6-2}.
Then we have
\[ h_0(\sF)(k) = \sF(k)/\sF(k)_{\rat}.
\]
\end{proposition}

\begin{proof}
By definition we have
\[ h_0(\sF)(k) 
= \Coker(i_0^* - i_{\infty}^* : \sF(\Delta) \to \sF(k)),
\]
where 
$\Delta = \P^1 \setminus \{ 1 \} (\cong \A^1)$ and
$i_a^*$ is the pull-back by the inclusion $i_a : \{ a \} \to \Delta$
for $a \in \{0, \infty \}$.

Suppose we are given a triple $(C, f, \alpha)$ as in \S \ref{sect:6-2},
and set $C' = f^{-1}(\Delta)$.
The graph $\gamma_{f|_{C'}}$ of $f|_{C'}$
defines an element of $c(\Delta, C') = L(C')(\Delta)$.
In the commutative diagram
\[
\begin{matrix}
L(C')(\Delta) & \overset{\alpha}{\to} & \sF(\Delta)
\\
{}_{i_0^*-i_{\infty}^*} \downarrow & & \downarrow_{i_0^*-i_{\infty}^*}
\\
L(C')(k) & \overset{\alpha}{\to} & \sF(k),
\end{matrix}
\]
the image of $\gamma_{f|_{C'}}$ in $L(C')(k) = Z_0(C')$
is $\div(f)$,
which shows the vanishing of $\alpha (\div(f))$ in $h_0(\sF)(k)$.

Conversely, given $\alpha \in \sF(\Delta)$,
\eqref{eq:relation}
for the triple $(\mathbf{P}^1, \id_{\mathbf{P}^1}, \alpha)$
coincides with $(i_0^*-i_{\infty}^*)(\alpha)$.
This completes the proof.
\end{proof}

\begin{lemma}\label{lem:6-2}
Let $\sF_1, \dots, \sF_n \in \PST$.
Put $\sF := \sF_1\otimes_{\PST} \dots \otimes_{\PST} \sF_n$.
Let $(C, f, \alpha)$ be a triple considered in \S \ref{sect:6-2}.
Then $\alpha \in \sF(C')$ is the sum of a finite number of elements of
the form
\begin{equation}\label{eq:tensordec}
 \Tr_{h}(g_1 \otimes \dots \otimes g_n),
\end{equation}
where $D$ is a smooth projective curve,
$h : D \to C$ is a surjective morphism,
$g_i \in \sF_i(h^{-1}(C'))$ for $i=1, \dots, n$,
and $\Tr_h : \sF(h^{-1}(C')) \to \sF(C')$
is the transfer with respect to $h|_{h^{-1}(C')}$.
\end{lemma}
\begin{proof}
By the facts recalled in \S \ref{sect:6-1},
we are reduced to the case 
$\sF_i = L(X_i)$ where $X_i$ is a smooth variety over $k$
for each $i=1, \dots, n$.
Then we have $\sF=L(X)$ with $X=X_1 \times \dots \times X_n$.
Let $Z$ be an integral closed subscheme of $C' \times X$
which is finite and surjective over $C'$.
It suffices to show that $Z \in c(C', X)=L(X)(C')$
can be written as \eqref{eq:tensordec}.

Let $q : D' \to Z$ be the normalization,
and let $h : D' \to C'$ be the composition $D' \to Z \to C'$,
so that $h$ is a finite surjective morphism.
For $i=1, \dots, n$,
we define $g_i \in c(D', X_i) = L(X_i)(D')$ 
to be the graph of $D' \to X \to X_i$.
If we set
$g = g_1 \otimes \dots \otimes g_n \in L(X)(D')$,
then by definition we have
$\Tr_h(g) = Z$ in $L(X)(C')$.
The assertion is proved.
\end{proof}

\begin{para}
Now it follows from Definition \ref{def:k'-groups} b),
Proposition \ref{prop:6-1}, 
Lemma \ref{lem:6-2} and \eqref{eq:relation'}
that \eqref{eq7} and \eqref{eq:h0} induce an isomorphism
\[
  K'(k; \sF_1, \dots, \sF_n) \cong
  h_0(\sF_1 \otimes_{\PST} \dots \otimes_{\PST} \sF_n)(k) 
\]
for any $\sF_1, \dots, \sF_n \in \PST$.
If $\sF_1, \dots, \sF_n \in \HI_\Nis$, the right hand side is
canonically isomorphic to
$\Hom_{\DM_-^\eff}(\Z, \sF_1[0]\otimes\dots\otimes \sF_n[0])$
by  \eqref{eq:dm} + \eqref{eq11}.
This completes the proof of Theorem \ref{thm:k'-group}. \qed
\end{para}

\section{Milnor $K$-theory}

\begin{para}
Let $\sF_1, \dots, \sF_n \in \HI_\Nis$.
We obtained a surjective homomorphism
\begin{equation}\label{eq:twokgroups}
   K(k; \sF_1, \dots, \sF_n) \to K'(k; \sF_1, \dots, \sF_n).
\end{equation}

Our aim is to show that this map is bijective. The first step is the special case of the multiplicative groups.
\end{para}

\begin{proposition}\label{prop:milnork}
When $\sF_1=\dots=\sF_n=\G_m$,
the map \eqref{eq:twokgroups} is bijective.
\end{proposition}

\begin{proof}
It suffices to show that relations \eqref{eq:rel-voevod}
vanish in  $K(k; \G_m, \dots, \G_m)$.
Because of Somekawa's isomorphism
\cite[Theorem 1.4]{somekawa}
\begin{equation}\label{eq:somekawaiso}
  K(k; \G_m, \dots, \G_m) \cong K_n^M(k)
\end{equation}
given by 
$\{ x_1, \dots, x_n \}_{E/k} \mapsto N_{E/k}(\{ x_1, \dots, x_n \})$,
it suffices to show this vanishing in the usual Milnor $K$-group $K_n^M(k)$,
which follows from Weil reciprocity \cite[Ch. I, (5.4)]{bt}.
\end{proof}


The following lemmas appear to be crucial in the proof of the main theorem.

\begin{lemma}\label{l7.3}
Let $C$ be a smooth projective connected curve over $k$,
and let $Z= \{ p_1, \dots, p_s \}$ be 
a finite set of closed points of $C$.
If $k$ is infinite, then
we have $K_2^M(k(C)) = \{ k(C)^*, \sO_{C,Z}^* \}$.
\end{lemma}

\begin{proof}
Let $\fp_i$ be the maximal ideal of $A=\sO_{C,Z}$ corresponding to $p_i$.
Since $A$ is a semi-local PID,
we can choose generators $\pi_1, \dots, \pi_s$ of $\fp_1,\dots,\fp_s$.
Since $k$ is infinite, we can change $\pi_i$ into $\mu_i\pi_i$ for suitable
$\mu_1,\dots,\mu_s\in k^*$ to achieve $\pi_i+\pi_j\not\equiv 0 \pmod{\fp_k}$ for $i,j,k$ all
distinct (indeed, the set of bad $(\mu_1,\dots,\mu_s)$ is contained in a finite union of
hyperplanes in $\bar k^s$). It follows that
$\pi_i +
\pi_j
\in A^*$ for all
$i\ne j$.

By the relation $\{f, -f\}=0 ~(f \in k(C)^*)$,
we have $K_2^M(k(C)) = \{ A^*, A^* \} + \sum_{i<j} \{ \pi_i, \pi_j \}$.
Now the identity
\[ \{ \pi_i, \pi_j \} 
  = \{ -\pi_i/\pi_j, \pi_i + \pi_j \}
\]
proves the lemma.
\end{proof}

\begin{lemma}\label{l7.1}
Let $C$ be a smooth projective connected curve over $k$,
and $r > 0$.
If $k$ is an infinite field,
then $K_{r+1}^M{k(C)}$ is generated by 
elements of the form
$\{a_1, \dots, a_{r+1} \}$
where the $a_i \in k(C)^*$ satisfy
$\Supp(\divi(a_i)) \cap \Supp(\divi(a_j)) = \emptyset$
for all $1 \leq i<j \leq r$.
\end{lemma}

\begin{proof}
We proceed by induction on $r$.
The assertion is empty when $r=1$.
Suppose $r>1$.
Take $a_1, \dots, a_{r+1} \in k(C)^*$.
By induction, there exist $b_{m, i} \in k(C)^*$
such that
$\Supp(\divi(b_{m, i})) \cap \Supp(\divi(b_{m, j})) = \emptyset$
for all $i<j<r$ and $m$, and 
\[ \{ a_1, \dots, a_r \}  
  = \sum_m \{ b_{m, 1}, \dots, b_{m, r} \}
\]
holds in $K_r^M k(C)$.
For each $m$, 
the above lemma shows that
there exist $c_{m, i}, d_{m, i} \in k(C)^*$ such that
\[
  \Supp(\divi(c_{m, i})) \cap 
  \left(\bigcup_{j=1}^{r-1} \Supp(\divi(b_{m, j}))\right) = \emptyset
\]
and that
\[ \{ b_{m, r}, a_{r+1} \} = \sum_i \{ c_{m, i}, d_{m, i} \} \]
holds in $K_2^M k(C)$.
Then we have 
\[ \{a_1, \dots, a_{r+1} \} = \sum_{m, i} 
    \{b_{m, 1}, \dots, b_{m, r-1}, c_{m, i}, d_{m, i} \}
\]
in $K_{r+1}^M k(C)$,
and we are done.
\end{proof}

\section{$K$-groups of Milnor type}\label{sect:milnor}

We now generalize the notion of Milnor $K$-groups to arbitrary homotopy invariant Nisnevich sheaves with transfers, although we shall seriously use this generalization only for special, representable, sheaves.

\begin{para}
Let $\sF\in \HI_\Nis$. We shall call a homomorphism $\G_m\to \sF$ a \emph{cocharacter} of
$\sF$. (By Proposition \ref{p2}, the group $\Hom_{\HI_\Nis}(\G_m,\sF)$ is canonically isomorphic
to $\sF_{-1}(k)$.)

Let $\sF_1,\dots,\sF_n\in \HI_\Nis$.
Denote by
$St(k;\sF_1,\dots,\sF_n)$ the subgroup of $(\sF_1\otimes_\PST\dots\otimes_\PST\sF_n)(k)$
generated by the elements
\begin{equation}\label{eq:steinbergrel}
 a_1 \otimes \dots  \otimes \chi_i(a)
   \otimes \dots \otimes \chi_j(1-a) \otimes \dots  \otimes a_n
\end{equation}
where 
$\chi_i:\G_m\to \sF_i$, $\chi_j:\G_m\to \sF_j$ are
$2$ cocharacters with $i<j$,
$a \in k^* \setminus \{1\}$,
and $a_m \in \sF_m(k)  ~(m \not= i, j)$.
\end{para}

\begin{defn}\label{d8.1} For $\sF_1,\dots,\sF_n\in \HI_\Nis$, we write $\tilde
K(k;\sF_1,\dots,\sF_n)$ for the quotient of $(\sF_1\oo^M\dots\oo^M\sF_n)(k)$ by
the subgroup generated by
$\Tr_{E/k}St(E;\sF_1,\dots,\sF_n)$, where $E$ runs through all finite extensions of $k$. This is the \emph{$K$-group of Milnor type} associated to $\sF_1,\dots,\sF_n$.
\end{defn}

\begin{para}\label{10.3} The assignment  $k\mapsto \tilde K(k;\sF_1,\dots,\sF_n)$ inherits the
structure of a cohomological Mackey functor, which is natural in
$(\sF_1,\dots,\sF_n)$. In particular, the choice of elements
$f_i\in
\sF_i(k)=\Hom_{\HI_\Nis}(\Z,\sF_i)$ for $i=1,\dots,r$ induces a homomorphism
\begin{multline}\label{eq10.1}
\tilde K(k;\sF_{r+1},\dots,\sF_n)=\tilde K(k;\Z,\dots,\Z,\sF_{r+1},\dots,\sF_n)\\
\to \tilde
K(k;\sF_1,\dots,\sF_n).
\end{multline}
\end{para}

\begin{lemma}
Let $\sF_1, \dots, \sF_n \in \HI_{\Nis}$.
The image of $St(k; \sF_1, \dots, \sF_n)$ vanishes
in $K(k;\sF_1,\dots,\sF_n)$.
Consequently, we have a surjective homomorphism 
$\tilde K(k;\sF_1,\dots,\sF_n)\to  K(k;\sF_1,\dots,\sF_n)$
and a composite surjective homomorphism
\begin{equation}\label{eq:threekgroups}
\tilde K(k;\sF_1,\dots,\sF_n)\Surj  K'(k;\sF_1,\dots,\sF_n)
\end{equation}
\end{lemma}

\begin{proof}
This is a straightforward generalization of 
Somekawa's proof of \cite[Th. 1.4]{somekawa}.
We need to show the image of \eqref{eq:steinbergrel}
vanishes in $K(k;\sF_1,\dots,\sF_n)$.
By functoriality, we may assume that
$\sF_i = \sF_j = \G_m$ for some $i<j$
and $\chi_i, \chi_j$ are the identity cocharacters.
Given $a_m \in \sF_m(k) ~(m \not= i, j)$
and $a \in k^* \setminus \{ 1 \}$,
we put $a_i = 1-at^{-1}, a_j = 1-t \in \G_m(k(\P^1)) = k(t)^*$.
Then $(\P^1, t, (a_1, \dots, a_n))$
is a relation datum of Somekawa type
and yields the vanishing of \eqref{eq:steinbergrel}.
\end{proof}

\begin{lemma}\label{p9.1}
Let $\sF_1, \dots, \sF_n \in \HI_{\Nis}$
and
let $\sG' \Surj \sG''$
be an epimorphism in $\HI_{\Nis}$.
If \eqref{eq:threekgroups} is bijective
for  $(\sG', \sF_1, \dots, \sF_n)$,
it is bijective for  $(\sG'', \sF_1, \dots, \sF_n)$.
\end{lemma}

\begin{proof}
Let $\sG=\Ker(\sG'\to\sG'')$. The induced sequence
\begin{multline*}
 \tilde K(k; \sG, \sF_1, \dots, \sF_n) \to
   \tilde K(k; \sG', \sF_1, \dots, \sF_n)\\ 
   \overset{(*)}{\longrightarrow}
   \tilde K(k; \sG'', \sF_1, \dots, \sF_n) \to 0
\end{multline*}
is a complex and $(*)$ is surjective.
The corresponding sequence for $K'$ is exact
because of Theorem \ref{thm:k'-group} and Lemma \ref{l2}.
The assertion follows by a diagram chase.
\end{proof}

\begin{lemma}\label{lem:projformula2}
Let $E/k$ be a finite extension.
Let $\sF_1, \dots, \sF_{n-1} \in \HI_\Nis$,
and let $\sF_n$ be a Nisnevich sheaf with transfers over $E$.
We have canonical isomorphisms
\begin{gather*}
 K(k; \sF_1, \dots, \sF_{n-1}, R_{E/k} \sF_n)
  \cong 
 K(E; \sF_1, \dots, \sF_n),
\\
 K'(k; \sF_1, \dots, \sF_{n-1}, R_{E/k} \sF_n)
  \cong 
 K'(E; \sF_1, \dots, \sF_n),
\\
 \tilde K(k; \sF_1, \dots, \sF_{n-1}, R_{E/k} \sF_n)
  \cong 
 \tilde K(E; \sF_1, \dots, \sF_n).
\end{gather*}
\end{lemma}
\begin{proof}
The first isomorphism was constructed
in \cite[Lemma 4]{sp-ya}
when $\sF_1, \dots, \sF_n$ are semi-abelian varieties.
The same construction works for arbitrary
$\sF_1, \dots, \sF_n$
and also for $K'$ and $\tilde K$.
\end{proof}

\begin{para} 
If
$\sF_1=\dots=\sF_n=\G_m$,
\eqref{eq:threekgroups} is bijective by Proposition
\ref{prop:milnork}. This is false in general, \emph{e.g.} if all the $\sF_i$ are proper
(Definition \ref{d7.1}) and
$n>1$. However, we have:
\end{para}

\begin{prop}\label{p10.1} a) Let $\sF_1=\sF'_1\oplus \sF''_1$. Then the natural map
\[\tilde K(k;\sF_1,\dots,\sF_n)\to \tilde K(k;\sF'_1,\dots,\sF_n)\oplus \tilde
K(k;\sF''_1,\dots,\sF_n)\] is bijective.\\
b) Let $T_1,\dots,T_n$ be tori. Assume that, for each $i$, there exists an exact sequence
\[0\to P_i^1\to P_i^0\to T_i\to 0\]
where $P_i^0$ and $P_i^1$ are invertible tori (i.e. direct summands of permutation tori). Then
\eqref{eq:threekgroups} is bijective.
\end{prop}

\begin{proof} a) This is formal, as  $\tilde K(k;\sF_1,\dots,\sF_n)$ is a quotient of the
multiadditive multifunctor $(\sF_1\oo^M\dots\oo^M\sF_n)(k)$ (see \ref{10.3}).

b) Note that, by Hilbert's theorem 90, the sequences $0\to P_i^1\to P_i^0\to T_i\to 0$ are exact in $\HI_\Nis$.  Lemma \ref{p9.1} reduces us to the case where all $T_i$ are permutation tori. Then Lemma \ref{lem:projformula2} reduces us to the case where all $T_i$ are split tori. Finally, we reduce to $\sF_1=\dots=\sF_n=\G_m$ by a).
\end{proof}

\begin{qn} \label{q7.1} Is proposition \ref{p10.1} true for general tori?
\end{qn}

\begin{para}\label{10.5} Let $T_1,\dots,T_n$ be as in Proposition \ref{p10.1} b); let $C/k$ be
a smooth projective connected curve, with function field $K$. From Proposition \ref{p10.1} b),
Theorem \ref{thm:k'-group}, Theorem \ref{t1} b) and \eqref{eq14}, we get a residue map
\[\partial_v:\tilde K(K;T_1,\dots,T_n,\G_m)\to \tilde K(k(v);T_1,\dots,T_n)\]
for any $v\in C$. These maps satisfy the reciprocity law of Proposition \ref{p3} and the compatibility of Lemma \ref{l4}.
\end{para}

\section{Reduction to the representable case}\label{repr}

Following \cite[p. 207]{voetri},
we write $h_0^{\Nis}(X) := h_0^{\Nis}(L(X))$
for a smooth variety $X$ over $k$.

\begin{proposition}\label{p8.1}
The following statements are equivalent:
\\
a) The homomorphism \eqref{eq:twokgroups} is bijective for any $\sF_1,\dots,\sF_n\in \HI_\Nis$.
\\
b) 
Let $\sF_1 = \dots = \sF_n = h_0^{\Nis}(C')$ for a smooth connected curve $C'/k$.
Then \eqref{eq:twokgroups} is bijective.
\\
c) 
Let $C$ be a smooth projective connected curve over $k$,
and let $f: C \to \mathbf{P}^1$ be a surjective morphism.
Let $C' = f^{-1}(\mathbf{P}^1 \setminus \{ 1 \})$
and let $\iota : C' \to \sA$ be the tautological morphism, where $\sA=h_0^\Nis(C')$.
These data define a relation datum of geometric type 
$(C, f, (\iota, \dots, \iota))$ for $\sF_1= \dots=\sF_n=\sA$,
and its associated element \eqref{eq:rel-voevod} is
\begin{equation}\label{eq:theta}
  \sum_{c\in C'} v_c(f) \Tr_{k(c)/k}
  (\iota(c) \otimes \dots\otimes \iota(c))
  \in \sA \oo^M \dots \oo^M \sA(k).
\end{equation}
Then the image of \eqref{eq:theta} in $K(k; \sA, \dots, \sA)$ vanishes.
\end{proposition}

\begin{proof}
Only the implication c) $\Rightarrow$ a) requires a proof.
Let $(C, f, (g_i))$ be a relation datum of geometric type
for $\sF_1, \dots, \sF_n$.
We need to show the vanishing of \eqref{eq:rel-voevod}
in $K(k; \sF_1, \dots, \sF_n)$. 

By adjunction, the section $g_i: M(C')\to \sF_i[0]$
induces a morphism $\phi_i : \sA \to \sF_i$
for all $i=1, \dots, n$.
Then
\begin{equation}\label{eq98}
  \sum_{c\in C'} v_c(f)
    \{ g_1(c) , \dots, g_n(c) \}_{k(c)/k} = 0
\quad \text{in}~K(k; \sF_1, \dots, \sF_n)
\end{equation}
because it is 
the image of  \eqref{eq:theta}
by the homomorphism 
$K(k; \sA, \dots, \sA) \to K(k; \sF_1, \dots, \sF_n)$
defined by $(\phi_1, \dots, \phi_n)$.
\end{proof}

\section{Proper sheaves}

\begin{definition}\label{d7.1}
Let $\sF$ be a Nisnevich sheaf with transfers.
We call $\sF$ \emph{proper} if,
for any smooth curve $C$ over $k$
and any closed point $c \in C$,
the induced map $\sF(\sO_{C, c}) \to \sF(k(C))$ is surjective. We say that $\sF$ is
\emph{universally proper} if the above condition holds when replacing $k$ by any finitely
generated extension $K/k$, and $C$ by any regular $K$-curve.
\end{definition}

\begin{example}
A semi-abelian variety $G$ over $k$ is proper (in the above sense)
if and only if $G$ is an abelian variety.
A birational sheaf $\sF \in \HI_\Nis$ in the sense of \cite{ks}
is by definition proper. If $C$ is a smooth proper curve, then $h_0^\Nis(C)$ is proper.
Other examples of birational sheaves 
will be given in Lemma \ref{lem:chowzero} b) below.
\end{example}

In fact:

\begin{lemma} Let $\sF\in \HI_\Nis$. Then\\
a) $\sF$ is proper if and only if $\sF(C)\iso \sF(k(C))$ for any smooth $k$-curve $C$.\\
b) $\sF$ is universally proper if and only if it is
birational in the sense of \cite{ks}.
\end{lemma}

\begin{proof} Let us prove b), as the proof of a) is a subset of it. Let $X$ be a smooth
$k$-variety. By
\cite[Cor. 4.19]{voepre}, the map
$\sF(X)\to
\sF(U)$ is injective for any dense open subset of $X$. By definition, $\sF$ is birational if
one may replace ``injective" by ``bijective". So birational $\Rightarrow$ universally proper.
Conversely, assume $\sF$ to be universally proper; let $x\in X^{(1)}$ and let $p:X\to
\A^{d-1}$ be a dominant rational map defined at
$x$, where $d=\dim X$. (We may find such a $p$ thanks to Noether's normalization
theorem.) Applying the hypothesis to the generic fibre of $p$, we find that $\sF(\sO_{X,x})\to
\sF(k(X))$ is surjective. Since this is true for all points $x\in X^{(1)}$, we
get the surjectivity of $\sF(X)\to \sF(k(X))$ from Voevodsky's Gersten resolution
\cite[Th. 4.37]{voepre}.
\end{proof}

The following proposition is not necessary for the proof of the main theorem, but its proof is much simpler than the general case.

\begin{proposition}\label{prop:propersheaves}
Let $\sF_1, \dots, \sF_n \in \HI_\Nis$.
Assume that $\sF_1, \dots, \sF_{n-1}$ are proper.
Then the homomorphism \eqref{eq:twokgroups} is bijective.
\end{proposition}

\begin{proof}
Suppose $(C, f, (g_i))$ is a relation datum of geometric type.
It suffices to show the element 
\eqref{eq:rel-voevod} vanishes in $K(k; \sF_1, \dots, \sF_n)$.
Let $\bar{g}_i$ be the image of $g_i$ in $\sF(k(C))$.
By assumption we have 
$\bar{g}_i \in \IM(\sF_i(\sO_{C,c}) \to \sF_i(k(C)))$ 
for all $c \in C$ and $i = 1, \dots, n-1$.
Hence $(C, h, (\bar{g}_i)_{i=1, \dots, n})$ is 
a relation datum of Somekawa type
(with $i(c)=n$ for all $c \in C$).
By Corollary \ref{c1},
the element \eqref{eq:rel-voevod} coincides
with \eqref{eq:rel-somekawa},
hence vanishes in $K(k; \sF_1, \dots, \sF_n)$.
\end{proof}

\section{Main theorem}\label{sect:isom}

\begin{defn} Let $\sF\in \HI_\Nis$. We say that $\sF$ is \emph{curve-like} if there exists an
exact sequence in $\HI_\Nis$
\begin{equation}\label{eq:curvelike}
0\to T\to \sF\to \bar \sF\to 0
\end{equation}
where $\bar \sF$ is proper (Definition \ref{d7.1}) 
and $T$ is a torus for which there exists an exact sequence
\begin{equation}\label{eq:specialtori}
0 \to R_{E_1/k}\G_m\to R_{E_2/k}\G_m \to T \to 0
\end{equation}
where $E_1$ and $E_2$ are \'etale $k$-algebras.
\end{defn}

This terminology is justified by the following lemma:

\begin{lemma}\label{lem:chowzero}
a)
If $C$ is a smooth curve over $k$,
then $h_0^\Nis(C)$ is the Nisnevich sheaf associated to the presheaf of relative Picard groups
\[U\mapsto \Pic(\bar C\times U,D\times U)\]
where $\bar C$ is the smooth projective completion of $C$, $D=\bar C \setminus C$ and $U$ runs
through smooth $k$-schemes.\\
b) If $X$ is a smooth projective variety over $k$,
then, for any smooth variety $U$ over $k$,
we have 
\begin{equation}\label{eq:cx-formula}
h_0^\Nis(X)(U) = CH_0(X_{k(U)}),
\end{equation}
where $k(U)$ denotes the total ring of fractions of $U$.
In particular, $h_0^\Nis(X)$ is birational.\\
c) For any smooth curve $C$, $h_0^\Nis(C)$ is curve-like.
\end{lemma}

\begin{proof}
a) and b) are proven in \cite[Th. 3.1]{sus-voe2}
and in \cite[Th. 2.2]{motiftate} respectively. 
With the notation of a), 
we put $E=H^0(\bar{C}, \sO_{\bar{C}})$.
Then c) follows from the exact sequence
\[0\to R_{E/k}\G_m\to R_{D/k}\G_m\to h_0^\Nis(C)\to h_0^\Nis(\bar C)\to 0\]
stemming from the Gysin exact triangle
\[M(D)(1)[1]\to M(C)\to M(\bar C)\by{+1}\]
of \cite[Prop. 3.5.4]{voetri}.
\end{proof}

\begin{remark} Let $\sF\in \HI_\Nis$ be curve-like.
The torus $T$ and proper sheaf $\bar{\sF}$ in \eqref{eq:curvelike}
are uniquely determined by $\sF$ up to unique isomorphism.
Indeed, this amounts to showing that any morphism $T \to \bar{\sF}$
is trivial. This is reduced to the case $T=R_{E/k} \G_m$ 
as in \eqref{eq:specialtori}, 
and further to $T=\G_m$ by adjunction as in Lemma \ref{lem:projformula2}.
Then we have 
$\Hom_{\HI_{\Nis}}(\G_m, \bar{\sF}) \cong \bar{\sF}_{-1}(k)=0$
by definition 
(see \eqref{eq:defofminusone} and Definition \ref{d7.1}).

We call $T$ and $\bar{\sF}$  the \emph{toric} and \emph{proper} part of $\sF$ respectively.
\end{remark}

\begin{lemma}\label{l10.2} a) Let $\sF\in \HI_\Nis$ be curve-like with toric part $T$, and let $C$ be a smooth
proper connected $k$-curve. Let $Z$ be a closed subset of $C$, $A=\sO_{C,Z}$ and $K=k(C)$. Then
the sequence
\[0\to T(A)\to T(K)\oplus \sF(A)\to \sF(K)\to 0\]
is exact.\\
b) Let $\sF_1,\dots,\sF_n\in \HI_\Nis$ be curve-like with toric parts $T_1, \dots, T_n$, and let $C,Z,A,K$ be as in a). Then the
group $\sF_1(K)\otimes\dots\otimes \sF_n(K)$ has the following presentation:
\begin{description}
\item[Generators] for each subset $I\subseteq \{1,\dots,n\}$, elements $[I;f_1,\dots, f_n]$
with $f_i\in \sF_i(A)$ if $i\in I$ and $f_i\in T_i(K)$ if $i\notin I$.
\item[Relations]\
\begin{itemize}
\item Multilinearity: 
\[[I;f_1,\dots,f_i+f'_i,\dots,f_n]=[I;f_1,\dots,f_i,\dots,f_n]+[I;f_1,\dots,f'_i,\dots,f_n].\]
\item Let $I\subsetneq \{1,\dots,n\}$ and let $i_0\notin I$. Let $[I;f_1,\dots,f_n]$ be a
generator. Suppose that $f_{i_0}\in T_{i_0}(A)$. Then
$[I;f_1,\dots,f_n]=[I\cup\{i_0\};f_1,\dots,f_n]$.
\end{itemize}
\end{description}
\end{lemma}

\begin{proof} a) Consider the commutative diagram of $0$-sequences
\[\begin{CD}
0@>>> T(A)@>>> \sF(A)@>>> \bar \sF(A)@>>> 0\\
&&@VVV @VVV @V{\wr}VV\\
0@>>> T(K)@>>> \sF(K)@>>> \bar \sF(K)@>>> 0.
\end{CD}\]

By [the proof of] \cite[Cor. 4.18]{voepre}, the top sequence is a direct summand of the bottom
one, which is clearly exact. Thus the top sequence is exact as well, and the lemma follows from
a diagram chase.  Then b) follows from a).
\end{proof}

\begin{prop}\label{p10.2} Let $C/k$ be a smooth proper connected curve, and let $v\in C, K=k(C)$.
Then there exists a unique law 
associating to a system
$(\sF_1,\dots,\sF_n)$ of $n$ curve-like sheaves a homomorphism
\[\partial_v:\sF_1(K)\otimes\dots\otimes\sF_n(K)\otimes K^*\to \tilde
K(k(v);\sF_1,\dots,\sF_n)\] such that
\begin{thlist}
\item If $\sigma$ is a permutation of $\{1,\dots,n\}$, the diagram
\[\begin{CD}
\sF_1(K)\otimes\dots\otimes\sF_n(K)\otimes K^*@>\partial_v>> \tilde K(k(v);\sF_1,\dots,\sF_n)\\
@V{\sigma}VV @V{\sigma}VV \\
\sF_{\sigma(1)}(K)\otimes\dots\otimes\sF_{\sigma(n)}(K)\otimes K^*@>\partial_v>> \tilde
K(k(v);\sF_{\sigma(1)},\dots,\sF_{\sigma(n)})
\end{CD}\]
commutes.
\item If $[I,f_1,\dots,f_n]$ is a generator of $\sF_1(K)\otimes \dots\otimes \sF_n(K)$ as in
Lemma \ref{l10.2} b) for some $Z$
containing $v$, with $I=\{1,\dots,i\}$, then
\[\partial_v(f_1\otimes\dots\otimes f_n\otimes f)= 
\{f_1(v),\dots f_i(v), 
\partial_v(\{ f_{i+1},\dots, f_n, f\}_{K/K})\}_{k(v)/k}\] where
$\partial_v(\{ f_{i+1},\dots, f_n, f\}_{K/K})$ is the residue of \ref{10.5}.
\end{thlist}
\end{prop}

\begin{proof} By Lemma \ref{l10.2} b), it suffice to check that $\partial_v$ agrees on
relations. Up to permutation, we may assume $I=\{1,\dots,i\}$ and $i_0=i+1$. The claim then
follows from Proposition \ref{l3}.
\end{proof}

\begin{lemma}\label{l11.2} a) Keep the notation of Proposition \ref{p10.2}. Let $L/K$ be a finite extension; write $D$ for the smooth projective model of $L$ and $h:D\to C$ for the corresponding morphism. Let $Z=h^{-1}(v)$. Write $\sF_{n+1} = \G_m$. Then, for any $i\in \{1,\dots,n+1\}$, the diagram
\[\xymatrix{
\sF_1(L)\otimes\dots\otimes\sF_{n+1}(L)\ar[r]^{(\partial_w)}& \displaystyle\bigoplus_{w\in Z}\tilde K(k(w);\sF_1,\dots,\sF_n)\ar[dd]^{(\Tr_{k(w)/k(v)})}\\
\sF_1(K)\otimes\dots\otimes\sF_i(L)\otimes \dots\otimes\sF_{n+1}(K)\ar[u]^u\ar[d]_d\\
\sF_1(K)\otimes\dots\otimes\sF_{n+1}(K)\ar[r]^{\partial_v}& \tilde K(k(v);\sF_1,\dots,\sF_n)
}\]
commutes, where $u$ is given componentwise by functoriality for $j\ne i$ and by the identity for $j=i$, and $d$ is given componentwise by the identity for $j\ne i$ and by $\Tr_{L/K}$ for $j=i$.\\
b) The homomorphisms $\partial_v$ induce
resi\-due maps
\[\partial_v:\left(\sF_1\oo^M\dots\oo^M\sF_n\oo^M\G_m\right)(K)\to \tilde K(k(v);\sF_1,\dots,\sF_n).\]
which verify the compatibility of Lemma \ref{l4} b).
\end{lemma}

\begin{proof} 
a) For clarity, we distinguish two cases: $i<n+1$ and $i=n+1$. 
In the former case, up to permutation we may assume $i=n$.  
It is enough to check commutativity 
on generators in the style of Lemma \ref{l10.2} b). 
Let $T_l$ denote the toric part of $\sF_l$.
In view of Lemma \ref{l10.2} a) and Proposition \ref{p10.2} (i),
it suffices to check the commutativity for 
$x=f_1\otimes\dots\otimes f_n\otimes f$
when one of the following two conditions is satisfied:

\begin{itemize}
\item[(i)]
for some $j \in \{ 0, \dots, n-1 \}$,
$f_l\in \sF_l(\sO_{C, v}) ~(1 \le l \le j)$,
$f_l\in T_l(K) ~(j+1 \le j \le n-1)$,
$f_n \in T_n(L)$ and $f\in K^*$.
\item[(ii)]
for some $j \in \{ 0, \dots, n-1 \}$,
$f_l\in \sF_l(\sO_{C, v}) ~~(1 \le l \le j)$,
$f_l\in T_l(K) ~(j+1 \le j \le n-1)$,
$f_n \in \sF_n(\sO_{D, Z})$ and $f\in K^*$.
\end{itemize}

Let $w\in Z$. 
If (i) holds, we have
\begin{align*}
\partial_w(u(x)) &= \{ f_1(w),\dots, f_j(w), 
\partial_w(\{f_{j+1},\dots, f_n, f\}_{L/L})\}_{k(w)/k(w)} \\
\intertext{and}
\partial_v(d(x)) &= \{f_1(v),\dots, f_j(v),
\partial_v(\{f_{j+1},\dots, \Tr_{L/K}(f_n), f\}_{K/K})\}_{k(v)/k(v)}. 
\end{align*}
Observe that
the restriction of $f_l(v)$ to $k(w)$ is $f_l(w)$
for every $w \in Z$ and $l=1, \dots, j$.
Since the residue maps $(\pd_w)$ of \S \ref{10.5}
verify the compatibility of Lemma \ref{l4},
the commutativity for $x$ follows.
(Recall that 
$\Tr_{k(w)/k(v)}( \{ a_1, \dots, a_n \}_{k(w)/k(w)})
=\{ a_1, \dots, a_n \}_{k(w)/k(v)}.$)

If (ii) holds, we have
\begin{align*}
\partial_w(u(x)) &=\{f_1(w),\dots, f_j(w),
\partial_w(\{f_{j+1},\dots, f_{n-1}, f\}_{L/L}), f_n(w)\}_{k(w)/k(w)}\\
\intertext{and}
\partial_v(d(x)) &=\{f_1(v),\dots, f_j(v),
\partial_v(\{f_{j+1},\dots\, f_{n-1}, f\}_{K/K}), \Tr_{L/K}(f_n)(v) \}_{k(v)/k(v)}.\end{align*}
In addition to the observation mentioned in (i),
we remark that the restriction of 
$\partial_v(\{f_{j+1},\dots, f_{n-1}, f\}_{K/K})$
to $k(w)$ is
$\partial_w(\{f_{j+1},\dots, f_{n-1}, f\}_{L/L})$
for every $w \in Z$.
The commutativity for $x$ follows
from Lemma \ref{l4} b) applied to $\sF_n$.

If $i=n+1$ the check is similar, the projection formula working on the last variable. 

Now b) follows from a) and the definition of $\oo^M$ as in \cite[p. 84]{decomposable}.
\end{proof}

\begin{lemma}\label{l11.1} The homomorphisms $\partial_v$ of 
Lemma \ref{l11.2}
induce
resi\-due maps
\[\partial_v:\tilde K(K;\sF_1,\dots,\sF_n,\G_m)\to \tilde K(k(v);\sF_1,\dots,\sF_n).\]
which verify the compatibility of Lemma \ref{l4} b).
\end{lemma}

\begin{proof} Set $\sF_{n+1}=\G_m$. Let $i<j$ be two elements of $\{1,\dots,n+1\}$ and let
$\chi_i:\G_m\to \sF_i$, $\chi_j:\G_m\to \sF_j$ be two cocharacters. Let $f\in K^*-\{1\}$. We
must show that $\partial_v$ vanishes on 
\[x=f_1\otimes\dots \otimes \chi_i(f)\otimes\dots \otimes\chi_j(1-f)\otimes\dots\otimes f_{n+1}\] 
for any $(f_1,\dots,f_{n+1})\in \sF_1(K)\times \dots
\times \sF_{n+1}(K)$ (product excluding $(i,j)$). By functoriality, we may assume that
$\chi_i,\chi_j$ are the identity cocharacters.  We distinguish two cases for clarity: $j<n+1$ and
$j=n+1$.
But exactly the same argument works for both cases.
Presently we suppose $j<n+1$. 

Up to permutation, we may assume $i=n-1$, $j=n$. Let us say that an element 
$(x_1,\dots,x_{n-2})\in \sF_1(K)\times\dots \times\sF_{n-2}(K)$ is \emph{in normal form} if, for
each
$i$, either
$x_i\in \sF_i(\sO_v)$ or $x_i\in T_i(K)$. Then Lemma
\ref{l10.2} reduces us to the case where
$(f_1, \dots, f_{n-2})$ is in normal
form. Up to permutation, we may assume that $f_r\in \sF_r(\sO_v)$ for $r\le r_0$ and $f_r\in
T_r(K)$ for $r_0<r\le n-2$. Then
\[\partial_v x= \{f_1(v),\dots, f_{r_0}(v),
   \partial_v(\{f_{r_0+1}, \dots, f_{n-2}, f, (1-f), f_{n+1}\}_{K/K})\}_{k(v)/k(v)}.
\]

Let $\phi_v:\tilde
K(k(v),T_{r_0+1},\dots,T_n)\to \tilde K(k(v),\sF_1,\dots ,\sF_n)$ be the homomorphism induced
by $(f_1(v),\dots, f_{r_0}(v))$ via \eqref{eq10.1}, and let
$\phi_K:T_{r_0+1}(K)\otimes\dots\otimes T_n(K)\otimes K^*\to  \sF_1(K)\otimes\dots
\otimes\sF_n(K)\otimes K^*$ be the analogous homomorphism defined by 
$(f_1,\dots,f_{r_0})$. The diagram
\[\begin{CD}
T_{r_0+1}(K)\otimes\dots\otimes T_n(K)\otimes K^*@>\partial_v>> \tilde
K(k(v);T_{r_0},\dots,T_n)\\
@V{\phi_K}VV @V{\phi_v}VV\\
\sF_1(K)\otimes\dots\otimes\sF_n(K)\otimes K^*@>\partial_v>> \tilde K(k(v);\sF_1,\dots,\sF_n)
\end{CD}\]
commutes. But the top map factors through $\tilde K(K;T_{r_0+1},\dots,T_n,\G_m)$, hence the
desired vanishing.

Thus we have shown that the map $\partial_v$ of Proposition \ref{p10.2} vanishes on
$St(K;\sF_1,\dots,\sF_n,\G_m)$. The conclusion now follows from Lemma \ref{l11.2} b).
\end{proof}

\begin{para} Let $\sF\in \HI_\Nis$ and let $C$ be a smooth proper $k$-curve. The \emph{support}
of a section $f\in \sF(k(C))$ is the finite set 
\[\Supp(f)=\{c\in C\mid \partial_c f\ne 0\}.\]
\end{para}

The following proposition generalizes Lemma \ref{l7.1}:

\begin{prop}\label{goodpositionlemma2}
Let $\sF_1,\dots,\sF_n$ be $n$ curve-like sheaves, and let 
$C$ be a smooth proper $k$-curve. Put $\sF_{n+1}=\G_m$.
If the  field $k$ is infinite, the group $\tilde
K(k(C);\sF_1,\dots,\sF_n,\G_m)$ is generated by elements
$\{f_1, \dots, f_{n+1} \}_{k(D)/k(C)}$
where $D$ is another curve, $D\to C$ is a finite surjective morphism and $f_i \in
\sF_i(k(D))$ satisfy
\begin{equation}\label{eq:goodpos}
\Supp(f_i) \cap \Supp(f_j) = \emptyset 
\quad \text{for all } 1 \leq i<j \leq n.
\end{equation}
\end{prop}

\begin{proof}
Lemma \ref{l10.2} b) reduces us to 
the case where all $\sF_i$ are
$R_{E_i/k} \G_m$ for some \'etale $k$-algebras $E_i/k$.
Using the formula
\[ (R_{E_1/k} \G_{m, E_1})_{E_2} 
  \cong R_{E_1 \otimes_k E_2/E_2} \G_{m, E_1 \otimes E_2} 
\]
and Lemma \ref{lem:projformula2} repeatedly,
we are further reduced to the case all $\sF_i$ are $\G_m$.
Then it follows from Lemma \ref{l7.1}.
\end{proof}

\begin{lemma}\label{lem:additionalproperties}
Let $C, D, \sF_1, \dots, \sF_n$ be  
as in Proposition \ref{goodpositionlemma2}.
Let $f_i \in \sF_i(k(D))$ and $v \in D$.
Put 
$\xi := \{ f_1, \dots, f_{n+1} \}_{k(D)/k(C)}$,
regarded as an element of 
$\tilde K(k(C); \sF_1, \dots, \sF_n, \G_m).$
\begin{enumerate}
\item
If $v(f_{n+1}-1)>0$,
then we have $\pd_v( \xi )=0.$
\item
Suppose \eqref{eq:goodpos} holds.
If $v \in \Supp(f_i)$ for some $1 \leq i \leq n$, then
we have
\[ \pd_v(\xi)
  = \{ f_1(v), \dots, \pd_v(f_i, f_{n+1}), \dots, f_n(v) \}_{k(v)/k}.
\]
\end{enumerate}
\end{lemma}

\begin{proof} This follows from Corollary \ref{c1} and Proposition \ref{l3}.
\end{proof}

\begin{prop}\label{cor:trivial}
Let $C$ be a smooth projective connected curve,
and let $\sF_1, \dots, \sF_n \in \HI_{\Nis}$ be curve-like.
The composition
\begin{multline*}
 \sum_{v \in C} \Tr_{k(v)/k} \circ \pd_v : 
\tilde K(k(C);\sF_1,\dots,\sF_n,\G_m) \to \tilde K(k; \sF_1, \dots, \sF_n)\\
\to K(k; \sF_1, \dots, \sF_n)
\end{multline*}
is the zero-map.
\end{prop}

\begin{proof} a) Assume first $k$ infinite.
If $\xi = \{f_1, \dots, f_{n+1}\}_{k(D)/k(C)}$ 
satisfies \eqref{eq:goodpos},
then we have
$\sum_{v \in C} \Tr_{k(v)/k} \circ \pd_v(\xi) = 0$
by Definition \ref{def:k-groups} and 
Lemma \ref{lem:additionalproperties} (2).
Hence the claim follows from Proposition \ref{goodpositionlemma2}.

b) If $k$ is finite, we use a classical trick: let $p_1,p_2$ be two distinct prime numbers,
and let $k_i$ be the $\Z_{p_i}$-extension of $k$. Let $x\in
\tilde K(k(C);\sF_1,\dots,\sF_n,\G_m)$. By a), the image of $x$ in $K(k; \sF_1, \dots, \sF_n)$
vanishes in $K(k_1; \sF_1, \dots, \sF_n)$ and $K(k_2; \sF_1, \dots, \sF_n)$, hence is $0$ by a
transfer argument.
 \end{proof}

Finally, we arrive at:

\begin{theorem}\label{tmain}
The homomorphism \eqref{eq1}
is an isomorphism for any $\sF_1,\dots,\sF_n\in \HI_\Nis$.
\end{theorem}
\begin{proof}
It suffices to show the statement in
Proposition \ref{p8.1} (3).
With the notation therein,
the image of \eqref{eq:theta} in $K(k; \sA, \dots, \sA)$ is 
seen to be 
$\sum_{v \in C} \Tr_{k(v)/k} \circ \pd_v 
 ( \{ \iota, \dots, \iota, f \}_{k(C)/k(C)} )$
by Lemma \ref{lem:additionalproperties},
hence trivial by
Proposition \ref{cor:trivial}.
\end{proof}

\section{Application to algebraic cycles}

\begin{para} 
We assume $k$ is of characteristic zero.
Let $X$ be a $k$-scheme of finite type,
and let $M^c(X):=C_*^c(X)\in \DM_-^\eff$ 
be the motive of $X$ with compact supports 
\cite[\S 4.1]{voetri}. 
Then the sheaf $\uCH_0(X)$ of \S \ref{s1.4} agrees with $H_0(M^c(X))$
by \cite[Th. 2.2]{motiftate}. 
If $X$ is quasi-projective,  we have an isomorphism
\[ CH_{-i}(X, j+2i) \cong \Hom_{\DM_-^\eff}(\Z, M^c(X))(i)[-j]) \]
for all $i \in \Z_{\geq 0}, j \in \Z$
by \cite[Prop. 4.2.9]{voetri}.\footnote{The proof of \emph{loc. cit.} is written for equidimensional schemes but is the same in general. Moreover, the assumption ``quasi-projective" can be removed if one replaces higher Chow groups by the Zariski hypercohomology of the cycle complex as in \cite[after Theorem 1.7]{levine}.}
\end{para}

\begin{proof}[Proof of Theorem \ref{prop:motivichom}]
Using Lemma \ref{l2},
we see 
\begin{align*}
   \Hom_{\DM^{\eff}_-}&(\Z, \uCH_0(X_1)[0] \otimes 
      \dots \otimes \uCH_0(X_n)[0]\otimes \G_m[0]^{\otimes r})
\\
  \cong& \Hom_{\DM^{\eff}_-}(\Z, M^c(X_1)  \otimes 
      \dots \otimes M^c(X_n) \otimes \G_m[0]^{\otimes r}).
\\
  \cong& \Hom_{\DM^{\eff}_-}(\Z, M^c(X)(r)[r])
  \cong CH_{-r}(X, r).
\end{align*}
(Here we used $\G_m[0] \cong \Z(1)[1]$.) Now the theorem follows from Theorem \ref{tmain}.
\end{proof}


\begin{para}\label{s8.1} Let $X$ be a $k$-scheme of finite type. 
Recall that for $i \in \Z_{\geq 0}, j \in \Z$
the motivic homology of $X$ is defined by
\cite[Def. 9.4]{fv} 
\begin{align}
\label{eq-def:higherchow}
H_j(X, \Z(-i)) := \Hom_{\DM^{\eff}_-}(\Z, M(X)(i)[-j]). 
\end{align}
When $i=0$,
$H_j(X, \Z(0))$ agrees with Suslin homology \cite{sus-voe2}.
\end{para}

\begin{theorem}\label{prop:higherchow}
Let $X_1,\dots,X_n$ be $k$-schemes of finite type. Suppose either the $X_i$ are smooth or $\car k=0$. Put $X=X_1 \times \dots \times X_n$.
For any $r \geq 0$, we have an isomorphism
\[K(k;h_0^\Nis(X_1),\dots,h_0^\Nis(X_n), \G_m,\dots,\G_m)\iso
H_{-r}(X, \Z(r)).
\]
\end{theorem}

\begin{proof}
Using Lemma \ref{l2},
we see 
\begin{align*}
   \Hom_{\DM^{\eff}_-}&(\Z, h_0^\Nis(X_1)[0] \otimes 
      \dots \otimes h_0^\Nis(X_n)[0]\otimes \G_m[0]^{\otimes r})
\\
  \cong& \Hom_{\DM^{\eff}_-}(\Z, M(X_1)  \otimes 
      \dots \otimes M(X_n) \otimes \G_m[0]^{\otimes r}).
\\
  \cong& \Hom_{\DM^{\eff}_-}(\Z, M(X)(r)[r])
  \cong H_{-r}(X, \Z(-r)).
\end{align*}

Now the theorem follows from Theorem \ref{tmain}.
\end{proof}

\begin{remark}\label{rem:sm-proj}
If $X_1, \dots, X_n$ are smooth projective varieties,
then \eqref{eq:cycle} is valid  in any characteristic.
Indeed, 
we have $M(X_i)=M^c(X_i)$ and hence $\uCH_0(X_i)=h_0^{\Nis}(X_i)$.
Moreover, 
\cite{allagree} and \cite[Appendix B]{motiftate}
show $H_{-r}(X, \Z(-r)) \cong CH_{-r}(X, r)$.
Thus \eqref{eq:cycle} follows from
Theorem \ref{prop:higherchow}.
\end{remark}

\appendix

\section{Extending monoidal structures}\label{stens}

\begin{para} Let $\sA$ be an additive category. We write $\sA\Mod$ for the category of
contravariant additive functors from $\sA$ to abelian groups. This is a Grothendieck abelian
category. We have the additive Yoneda embedding
\[y_\sA:\sA\to \sA\Mod\]
sending an object to the corresponding representable functor.
\end{para}

\begin{para}\label{Aadj} Let $f:\sA\to \sB$ be an additive functor. We have an induced functor
$f^*:\sB\Mod\to \sA\Mod$ (``composition with $f$"). As in \cite[Exp. 1, Prop. 5.1 and
5.4]{sga4}, the functor
$f^*$ has a left adjoint
$f_!$ and a right adjoint $f_*$ and the diagram
\[\begin{CD}
\sA@>y_\sA>> \sA\Mod\\
@VfVV @Vf_!VV\\
\sB@>y_\sB>> \sB\Mod
\end{CD}\]
is naturally commutative.
\end{para}

\begin{para} If $f$ is fully faithful, then $f_!$ and $f_*$ are fully faithful and $f^*$ is a
localization, as in \cite[Exp. 1, Prop. 5.6]{sga4}. 
\end{para}

\begin{para} Suppose that $f$ has a left adjoint $g$. Then we have natural isomorphisms
\[g^*\simeq f_!,\quad g_*\simeq f^*\]
as in \cite[Exp. 1, Prop. 5.5]{sga4}.
\end{para}

\begin{para} Suppose further that $f$ is fully faithful. Then $g^*\simeq f_!$ is fully faithful.
From the composition 
\[g^*g_*\Rightarrow Id_{\sA\Mod}\Rightarrow g^*g_!\]
of the unit with the counit, one then deduces a canonical morphism of functors
\[g_*\Rightarrow g_!.\]
\end{para}

\begin{para}\label{Atens} Let $\sA$ and $\sB$ be two additive categories. Their \emph{tensor
product} is the category $\sA\boxtimes \sB$ whose objects are finite collections $(A_i,B_i)$
with
$(A_i,B_i)\in \sA\times\sB$, and
\[(\sA\boxtimes\sB)((A_i,B_i),(C_j,D_j)) = \bigoplus_{i,j}\sA(A_i,C_j)\otimes \sB(B_i,D_j).\]

We have a ``cross-product" functor
\[\boxtimes:\sA\Mod\times \sB\Mod\to (\sA\boxtimes\sB)\Mod\]
given by
\[(M\boxtimes N)((A_i,B_i)) = \bigoplus_i M(A_i)\otimes N(B_i).\]
\end{para}

\begin{para}\label{A!} Let $\sA$ be provided with a biadditive bifunctor $\bullet:\sA\times
\sA\to
\sA$. We may view $\bullet$ as an additive functor $\sA\boxtimes \sA\to \sA$. We may then extend
$\bullet$ to $\sA\Mod$ by the composition
\[\sA\Mod\times \sA\Mod\by{\boxtimes}(\sA\boxtimes \sA)\Mod\by{\bullet_!}\sA\Mod.\]
This is an extension in the sense that the diagram
\[\begin{CD}
\sA\times \sA@>y_\sA\times y_\sA>> \sA\Mod\times \sA\Mod\\
@V\bullet\times\bullet VV @V\bullet VV\\
\sA@>y_\sA>> \sA\Mod 
\end{CD}\]
is naturally commutative.

If $\bullet$ is monoidal (resp. monoidal symmetric), then its associativity and commutativity
constraints canonically extend to $\sA\Mod$.
\end{para}

\begin{para}\label{Acoh} Let $\sA,\sB$ be two  additive symmetric monoidal categories, and let
$f:\sA\to
\sB$ be an additive symmetric monoidal functor. The above definition shows that the functor
$f_!:\sA\Mod\to \sB\Mod$ is also symmetric monoidal.
\end{para}

\begin{para} In \S\ref{A!}, let us write $\bullet_!= \int$ for clarity. Let $P\in (\sA\boxtimes
\sA)\Mod$. Then $\int P$ is the \emph{left Kan extension of $P$ along $\bullet$} in the sense
of \cite[X.3]{mcl}. This gives a formula for $\int P$ as a \emph{coend} (ibid., Theorem X.4.1); for
$A\in\sA$:
\begin{equation}\label{eqA.1}
\int P(A) = \int^{(B,B')} \sA(A,B\bullet B')\otimes P(B,B').
\end{equation}

In particular:
\end{para}

\begin{prop} Suppose $\sA$ rigid. Then \eqref{eqA.1} simplifies as
\[\int P(A) = \int^B  P(B, A\bullet B^*)\]
where $B^*$ is the dual of $B\in \sA$. In particular, if $P=M\boxtimes N$ for $M,N\in \sA\Mod$,
we have for $A\in\sA$:
\begin{equation}\label{eqA.2}
(M\bullet N)(A)= \int^B  M(B)\otimes N(A\bullet B^*)
\end{equation}
which describes $M\bullet N$ as a ``convolution".
\end{prop}

\begin{proof} Applying \eqref{eqA.1} and rigidity, we have
\begin{multline*}
\int P(A) = \int^{(B,B')} \sA(A,B\bullet B')\otimes P(B,B')\\
=\int^{(B,B')} \sA(A\bullet B^*, B')\otimes P(B,B')\\
=\int^B  P(B,A\bullet B^*)
\end{multline*}
because in the third formula, the variable $B'$ is dummy (this simplification is not in Mac
Lane!).
\end{proof}



\begin{thebibliography}{SGA4}
\bibitem{akhtar} R. Akhtar, 
Milnor $K$-theory of smooth varieties. 
$K$-Theory {\bf 32} (2004), no. 3, 269--291. 


\bibitem{bar-kahn} L. Barbieri-Viale, B. Kahn {\it On the derived category of $1$-motives},
preprint, 2010, {\tt arXiv:1009.1900}.

\bibitem{bt} H. Bass, J. Tate {\it The Milnor ring of a global field}, Lect. Notes in Math.
{\bf 342}, 349--446, Springer, 1973.

\bibitem{bloch} S. Bloch,
{\it Algebraic cycles and higher K-theory},
Adv. in Math. 61 (1986), no. 3, 267--304. 

\bibitem{deglise} F. D\'eglise Modules homotopiques avec transferts et motifs g\'en\'eriques,
PhD thesis, Paris 7, 2002.
\bibitem{fv} E. Friedlander, V. Voevodsky;
{\it Bivariant cycle cohomology}, 
in Cycles, transfers, and motivic homology theories, 138--187, 
Ann. of Math. Stud., 143, Princeton Univ. Press, Princeton, NJ, 2000. 
\bibitem{motiftate} A. Huber and B. Kahn {\it The slice filtration and mixed Tate motives},
Compositio Math. {\bf 142} (2006), 907--936.
\bibitem{decomposable} B. Kahn {\it The decomposable part of motivic cohomology and bijectivity
of the norm residue homomorphism}, Contemp. Math. {\bf 126} (1992), 79--87.
\bibitem{knote} B. Kahn {\it Nullit\'e de certains groupes attach\'es aux vari\'et\'es
semi-ab\'eliennes sur un corps fini; application},  C. R. Acad. Sci. Paris {\bf 314} (1992),
1039--1042.
\bibitem{ks} B. Kahn, R. Sujatha;
{\it Birational motives, I} (preliminary version),  {\tt www.math.uiuc.edu/K-theory} \# {\bf 596}, 2002.
\bibitem{somprem} B. Kahn {\it Somekawa's $K$-groups and Voevodsky's Hom groups} (preliminary
version), {\tt arXiv:1009.4554v1}, 2010.
\bibitem{ks2} K. Kato, S. Saito {\it Unramified class field theory of arithmetical surfaces}, Ann. of Math. (2) {\bf 118} (1983), 241--275.
\bibitem{kelly} S. Kelly Trinagulated categories of motives in positive characteristic, PhD thesis, Paris 13, October 2012.
\bibitem{levine} M. Levine {\it Techniques of localization in the theory of algebraic cycles}, 
J. Alg. Geom. {\bf 10} (2001) 299--363.
\bibitem{mcl} S. Mac Lane Categories for the working mathematician,
Springer, 1971.
\bibitem{mvw} C. Mazza, V. Voevodsky, C. Weibel;
Lecture notes on motivic cohomology.
Clay Mathematics Monographs {\bf 2}, American Mathematical Society, Providence, RI; Clay
Mathematics Institute, Cambridge, MA, 2006.
\bibitem{mochi} S. Mochizuki {\it  Motivic interpretation of Milnor $K$-groups attached to Jacobian varieties}, Hokkaido Mathematical Journal, 41 (2012), 1--10.
\bibitem{ram} N. Ramachandran, 
{\it Duality of Albanese and Picard $1$-motives,} 
$K$-Theory {\bf 21} (2001), 271--301.
\bibitem{rs} W. Raskind, M. Spie\ss, 
Milnor $K$-groups and zero-cycles on products of curves 
over $p$-adic fields. 
Compositio Math. {\bf 121} (2000), no. 1, 1--33.
\bibitem{ss2} A. Schmidt, M. Spie\ss,
Singular homology and class field theory of varieties over finite fields.
J. Reine Angew. Math.  527  (2000), 13--36.
\bibitem{gacl} J.-P. Serre Groupes alg\'ebriques et corps de classes, Hermann, 1959.
\bibitem{somekawa} M. Somekawa {\it On Milnor $K$-groups attached to semi-abelian varieties}, 
$K$-Theory {\bf 4} (1990), 105--119.
\bibitem{spsz} M. Spie\ss, T. Szamuely {\it On the Albanese map for smooth quasiprojective
varieties}, Math. Ann. {\bf 235} (2003), 1--17.
\bibitem{sp-ya} M. Spie\ss, T. Yamazaki {\it A counterexample to generalizations of the
Milnor-Bloch-Kato conjecture}, J. K-Theory {\bf 4} (2009), 77--90. 
\bibitem{suslin} A. Suslin {\it Higher Chow groups and \'etale cohomology}, {\it in}
Cycles, transfers, and motivic homology theories,  Ann. of Math. Stud. {\bf 143}, Princeton
Univ. Press, 2000, 239--254.
\bibitem{sus-voe} A. Suslin, V. Voevodsky
{\it Bloch-Kato conjecture and motivic cohomology with finite coefficients},
The arithmetic and geometry of algebraic cycles (Banff, AB, 1998), 117--189,
NATO Sci. Ser. C Math. Phys. Sci. {\bf 548}, Kluwer Acad. Publ., Dordrecht, 2000.
\bibitem{sus-voe2} A. Suslin, V. Voevodsky
{\it Singular homology of abstract algebraic varieties},
Invent. Math. {\bf 123} (1996), no. 1, 61--94.
\bibitem{thevenaz} J. Th\'evenaz {\it A visit to the kingdom of Mackey functors}, Bayreuth.
Math. Schriften {\bf 33} (1990), 215--241.
\bibitem{voepre} V. Voevodsky {\it Cohomological theory of presheaves with transfers}, {\it
in} Cycles, transfers, and motivic homology theories,  Ann. of Math. Stud. {\bf 143}, Princeton
Univ. Press, 2000, 87--137.
\bibitem{voetri} V. Voevodsky {\it Triangulated categories of motives over a field}, {\it in}
Cycles, transfers, and motivic homology theories,  Ann. of Math. Stud. {\bf 143}, Princeton
Univ. Press, 2000, 188--238.
\bibitem{allagree} V. Voevodsky {\it Motivic cohomology groups are isomorphic to higher Chow groups in any characteristic}, IMRN {\bf 2002},  351--355.
\bibitem{voecan} V. Voevodsky {\it Cancellation theorem}, Doc. Math. Extra Volume: Andrei A.
Suslin's Sixtieth Birthday (2010), 671--685.
\bibitem[SGA4]{sga4} M. Artin, A. Grothendieck, J.-L. Verdier Th\'eorie des topos
et cohomologie \'etale des sch\'emas (SGA4) (1\`ere partie), Lect. Notes in Math.
{\bf 269}, Springer, 1972.
\end{thebibliography}
\end{document}